\documentclass{article}
\usepackage{amsmath,amssymb,amsfonts,amsthm}
\usepackage{mathtools,mathrsfs,url}
\usepackage{color}
\usepackage{cleveref}
\usepackage[a4paper]{geometry}
\usepackage{algorithm}
\usepackage{algpseudocode}
\usepackage{tikz}
\usepackage{pgfplots}
\usepackage{pgfplotstable}
\usepackage{xspace}
\usepackage{enumerate}
\pgfplotsset{compat=1.16, width=.7\textwidth}

\newcommand{\norm}[1]{\lVert #1 \rVert}
\newcommand{\R}{\ensuremath{\mathbb{R}}}
\newcommand{\C}{\ensuremath{\mathbb{C}}}

\newcommand{\rat}{\mathcal{Q}}

\newcommand{\pol}{\ensuremath{\mathbb{P}}}
\newcommand{\T}{\ensuremath{\mathcal{T}}}

\newcommand{\blkdiag}{\mathrm{blkdiag}}

\theoremstyle{plain}
\newtheorem{lemma}{Lemma}[section]
\newtheorem{theorem}[lemma]{Theorem}

\newtheorem{proposition}[lemma]{Proposition}
\newtheorem{remark}[lemma]{Remark}
\newtheorem{definition}[lemma]{Definition}

\newcommand{\CKR}{{\tt TelFun}\xspace}
\newcommand{\CKM}{{\tt CKM}\xspace}

\renewcommand{\vec}[1]{\boldsymbol{#1}}		

 \author{
     Angelo A. Casulli\thanks{
         Scuola Normale Superiore, Pisa, Italy
         (\texttt{angelo.casulli@sns.it}).
     }\and
     Daniel Kressner\thanks{
         Institute of Mathematics, EPFL Lausanne, Lausanne, Switzerland
         (\texttt{daniel.kressner@epfl.ch}).
     }
     \and
     Leonardo Robol\thanks{
        Department of Mathematics, 
        University of Pisa, Pisa, Italy
        (\texttt{leonardo.robol@unipi.it})
     }
 }
\title{Computing Functions of \\ Symmetric Hierarchically Semiseparable Matrices}

\begin{document}
    \maketitle 

    \begin{abstract}
The aim of this work is to develop a fast algorithm for approximating the matrix function $f(A)$ of a square matrix $A$ that is symmetric and has hierarchically semiseparable (HSS) structure.  Appearing in a wide variety of applications, often in the context of discretized (fractional) differential and integral operators, HSS matrices have a number of attractive properties facilitating the development of fast algorithms. In this work, we use an unconventional telescopic decomposition of $A$, inspired by recent work of Levitt and Martinsson on approximating an HSS matrix from matrix-vector products with a few random vectors. This telescopic decomposition allows us to approximate $f(A)$ by recursively performing low-rank updates with rational Krylov subspaces while keeping the size of the matrices involved in the rational Krylov subspaces small. In particular, no large-scale linear system needs to be solved, which yields favorable complexity estimates and reduced execution times compared to existing methods, including an existing divide-and-conquer strategy. The advantages of our newly proposed algorithms are demonstrated for a number of examples from the literature, featuring the exponential, the inverse square root, and the sign function of a matrix.
Even for matrix inversion, our algorithm exhibits superior performance, even if not specifically designed for this task.
    \end{abstract}

\section{Introduction}
Consider a symmetric matrix $A\in\R^{n\times n}$ with spectral decomposition $A=V\Lambda V^T$, with the orthogonal matrix $V$ and the diagonal matrix $\Lambda=\mathrm{diag}(\lambda_1,\dots,\lambda_n)$ containing the eigenvalues of $A$. Given a scalar function $f$ well defined on the eigenvalues of $A$, the matrix function $f(A) \in\R^{n\times n}$ is defined as $Vf(\Lambda)V^T$, where $f(\Lambda):=\mathrm{diag}(f(\lambda_1),\dots,f(\lambda_n))$. Popular examples include the matrix inverse, the matrix exponential, the sign function, and the (inverse) matrix square root; see the monograph~\cite{higham2008functions} for an overview. When $A$ is of moderate size, $f(A)$ can simply be computed according to this definition, via computing the spectral decomposition of $A$, or using a more specialized algorithm such as the scaling-and-squaring method for the matrix exponential~\cite{higham2005scaling}. These methods typically require $\mathcal O(n^2)$ memory and $\mathcal O(n^3)$ operations, and thus become infeasible for larger $n$.
If only the computation of $f(A)B$ for a (block) vector $B$ is needed, (rational) Krylov subspace methods are well suited when $A$ is large and (data) sparse; see~\cite{guttel2010rational} and the references therein.

The task of approximating the whole matrix function $f(A)$ for a large-scale matrix $A$ is rather challenging and certainly requires additional assumptions on the data sparsity structure of $A$. For example, if $A$ is a banded matrix \emph{and} $f$ can be well approximated by a low-degree polynomial on the spectrum of $A$, then $f(A)$ can also be well approximated by a banded matrix~\cite{benzi2015decay}, leading to fast algorithms, such as the ones described in~\cite{cortinovis2022divide,frommer2021analysis, park2023approximating}. 
If, on the other hand, $f$ does not admit good polynomial approximations then $f(A)$ usually does not admit a good approximation by a banded or, more generally, by a sparse matrix even when $A$ is banded. Examples include the (inverse) square root or the sign function when $A$ has eigenvalues that are close to zero relative to the width of the spectrum. For these examples, $f$ still admits good \emph{rational} approximations and the approximation of $f(A)$ can potentially be addressed using hierarchical low-rank techniques~\cite{hackbusch2015hierarchical}.

A matrix $A$ is said to be \emph{hierarchically off-diagonal low-rank} (HODLR) if it can be recursively block partitioned in a matrix with low-rank off-diagonal blocks, more specifically, there exists a block partitioning
\begin{equation*}
    A=\begin{bmatrix}
        A_{11} & A_{12}\\ A_{21} & A_{22}
    \end{bmatrix}  
\end{equation*} 
such that $A_{12}$ and $A_{21}$ are of low rank and $A_{11}$ and $A_{12}$ are square matrices that can be recursively partitioned in the same way (until a minimal size of the diagonal blocks is reached). \emph{Hierarchically semiseparable} (HSS) matrices additionally impose that the low-rank factors representing the off-diagonal blocks on the different levels of the recursion are nested; see Section~\ref{sec:HSS} for the precise definition. Hierarchical matrices, such as HODLR and HSS, admit data-sparse representations and cover a wide variety of matrix structures, including banded matrices and rational functions thereof. In particular, they can be used to approximate $f(A)$ when $A$ is banded or, more generally, HSS whenever $f$ admits a good rational approximation~\cite{cortinovis2022divide}. This property has been exploited to develop iterate-and-truncate methods using hierarchical matrices~\cite{Grasedyck2003,KressnerSus2017} as well as a divide-and-conquer procedure based on low-rank updates and rational Krylov subspaces~\cite{cortinovis2022divide}. While numerical experiments in~\cite{cortinovis2022divide} show that the latter method is often preferable in terms of efficiency, it only exploits HODLR structure even when the matrix is HSS, and therefore does not fully benefit from the nestedness of low-rank factors in the HSS format. 

To fully benefit from HSS structure, we will represent an HSS matrix in an unconventional way, via a telescopic decomposition. Such a decomposition was used by Levitt and Martinsson~\cite{levitt2022linear} to compute an HSS approximation of a matrix $A$ from a few random matrix-vector products.\footnote{It is worth mentioning that another method for approximating HSS matrices from random matrix-vector products was recently presented by Halikias and Townsend~\cite{halikiasstructured}, but this algorithm is not based on the telescopic decompositions considered in this work.} 
For the purpose of this work, we consider a more general \emph{telescopic decomposition} (see Section~\ref{sec:telescopic-decomposition}), based on representing an HSS matrix $A$ as 
\begin{equation}\label{eqn:telesc_intro}
    A=\vec U\tilde{A}\vec V^T + \vec D,
\end{equation}
where $\vec U$ and $\vec V$ are block diagonal, orthonormal matrices with tall and skinny diagonal blocks, $\vec D$ is a block diagonal matrix and $\tilde{A}$ is a (smaller) square matrix recursively decomposed in the same fashion. This representation is not unique and includes the one from~\cite{levitt2022linear}. We also show how to convert such telescopic decompositions into the standard representation of HSS matrices.

The main contribution of this work is to design a new algorithm that returns a telescopic decomposition for an HSS approximation of $f(A)$, where $A$ is a symmetric HSS matrix and $f$ admits a good rational approximation. Unlike the divide-and-conquer method described in \cite{cortinovis2022divide}, our algorithm fully exploits the nestedness relations of the HSS format and, in turn, only needs to build rational Krylov subspaces for small-sized matrices. This translates into reduced complexity and, in many cases, significantly lower execution times.
Using telescopic decompositions, our new algorithm combines well with the method by Levitt and Martinsson. This combination allows one to extract an approximation to $f(A)$ from the product of $A$ with a few random vectors.

The remainder of the paper is structured as follows, In Section~\ref{sec:Krylov}, we briefly recall block rational Krylov subspaces. While the usual data-sparse representation of HSS matrices is described in Section~\ref{sec:HSS}, Section~\ref{sec:telescopic-decomposition} introduces a general concept of telescopic decompositions and shows how to convert between different forms of telescopic decompositions and the usual data-sparse representation of HSS matrices.  In Section~\ref{sec:comp-fun}, our newly proposed algorithm for approximating matrix functions is described and its convergence is analyzed. Section~\ref{sec:pole_selection} recalls existing rational approximation results, which yield favorable complexity estimates for our algorithm in important special cases. This translates into superior efficiency, as demonstrated with several numerical examples in Section~\ref{sec:Numerical_exp}.

Note that for most of Sections~\ref{sec:Krylov} to~\ref{sec:telescopic-decomposition}, we will not assume that $A$ is symmetric. Imposing symmetry would not simplify the exposition and our considerations on telescopic decompositions in Section~\ref{sec:telescopic-decomposition} might be of independent interest. In Section~\ref{sec:symmetric}, we will discuss the consequences of $A$ being symmetric.
    
    \section{Rational Krylov subspaces}\label{sec:Krylov}
    
   We briefly recall the definition of (block) rational Krylov subspaces; see~\cite{elsworth2020block} for more details.
    \begin{definition} \label{def:ratkrylov}
        Let $A\in \R^{n\times n}$, $B\in \R^{n\times b}$ and let $\vec{\xi}_k = ( \xi_0,\dots,\xi_{k-1})$, with $\xi_j \in \C\cup\{\infty\}$, be a list of poles that are not eigenvalues of $A$. Then the associated rational Krylov subspace is defined as
        \begin{equation*}
            \rat(A,B,\vec \xi_k):=\Big\{q(A)^{-1}\sum_{j=0}^{k-1}A^jBC_j,\quad  \text { with } C_j\in\R^{b\times b}\Big\},
        \end{equation*}  
        where 
        \begin{equation*}q(x):=\prod_{\xi\in \vec \xi_k, \xi\neq \infty}(x-\xi).
        \end{equation*}
\end{definition}
An orthonormal basis of a rational Krylov subspace can be computed by the \emph{rational Arnoldi method} \cite[Algorithm~2.1]{elsworth2020block}. During the procedure, shifted linear systems with the matrices\footnote{The linear system is replaced by a matrix product with $A$ if $\xi_j = \infty$.} $A-\xi_j I$ need to be solved and this can become computationally expensive for a large-scale matrix $A$. One novelty of this work is to employ rational Krylov subspaces associated with small-size matrices for the computation of large-scale matrix functions.

Proposition~\ref{prop:blkdiag_krylov} below is a simple but key result for the development of this paper. Given a family of matrices $(C_i \in \R^{m_i \times n_i})$, we let $\vec C =\blkdiag(C_i)$ denote the block diagonal matrix containing the matrices $C_i$ as diagonal blocks (in their natural order).
\begin{proposition}\label{prop:blkdiag_krylov}
  Let $\vec A:=\blkdiag(A_i)$ and $\vec B:=\blkdiag(B_i)$ with $A_i\in \R^{n_i\times n_i}$ and $B_i\in \R^{n_i\times b_i}$ and let $\vec \xi_k$ be a list of poles. Then an orthonormal basis of $\rat(\vec A,\vec B,\vec \xi_k)$ is given by 
$
    \vec U:=\blkdiag (U_i),
$
where $U_i$ is an orthonormal basis of $\rat(A_i,B_i,\vec \xi_k)$.
\end{proposition}
\begin{proof}
By Definition~\ref{def:ratkrylov}, every element of $\rat(\vec A, \vec B, \vec \xi_k)$ takes the form
\begin{equation*}
    q(\vec A)^{-1}\sum_{j=0}^{k-1}\vec A^j \vec BC_j=\sum_{j=0}^{k-1}\blkdiag\big(q(A_i)^{-1} A_i^j B_i\big)C_j.
\end{equation*}
Because $q(A_i)^{-1} A_i^j B_i\in \rat(A_i,B_i,\vec \xi_k)$ for every $i, j$, it follows that $\rat(\vec A, \vec B, \vec \xi_k)=\mathrm{span}(\blkdiag(U_i)).$ 
\end{proof}
 
   \section{Hierarchically semiseparable matrices}\label{sec:HSS}
    
    To define hierarchically semiseparable (HSS) matrices we need to introduce a way to recursively split row and column indices of a matrix. Given a vector of indices $I=[1,2,\dots,n],$ we use a perfect binary tree $\T,$ called $\emph{cluster tree}$, to define subsets of indices obtained by subdividing $I$. The root $\gamma$ of the tree is associated with the full vector $I$; the rest of the tree is recursively defined in the following way: given a non-leaf node $\tau$ associated with an index vector $I_{\tau}$, its children $\alpha,\beta$, are associated with two vectors of consecutive indices $I_{\alpha},I_{\beta}$, such that $I_{\tau}$ is the concatenation of $I_\alpha$ and $I_\beta$. The depth of a node is defined as the distance from the root of the tree. The depth of the tree is denoted by $L$. We observe that $\T$ is uniquely defined by the index vectors associated with the nodes and, hence, cluster trees can be defined by simply specifying the indices associated with the leaf nodes.
    For each leaf node $\alpha$, we use $|\alpha|$ to denote the length of the vector of indices associated with $\alpha$. We assume that the leaf nodes are ``small'', that is, for some prescribed \emph{threshold size} $t$ it holds that $|\alpha| \le t$ for every leaf $\alpha$.

    Given a matrix $A\in \R^{n \times n}$, we let $A_{\tau,\tau'}$ denote the submatrix of $A$ obtained 
    by selecting the row and column indices associated with nodes $\tau$ and $\tau'$, respectively. In particular, we will consider diagonal blocks ($\tau = \tau'$) and sub/supdiagonal blocks ($\tau$ and $\tau'$ are siblings, i.e., children of the same node).  We are now ready to state the definition of an HSS matrix following \cite[Section~3.3]{martinsson2011fast}.

    \begin{definition}\label{def:HSS}
    Given a cluster tree $\mathcal{T}$ for the indices $[1,\dots,n]$,
        a matrix $A\in \R^{n\times n}$ is called an HSS matrix of HSS rank $r$ if:
        \begin{enumerate}       
        \item for each pair of sibling nodes $\tau, \tau' \in \mathcal{T}$, there exist matrices $U^{(\mathrm{big})}_{\tau} \in \R^{|\tau|\times r}$, $V^{(\mathrm{big})}_{\tau'} \in \R^{|\tau'|\times r}$ with orthonormal columns and $\tilde{A}_{\tau,\tau'}\in\R^{r\times r}$, such that 
        \begin{equation*}
            A_{\tau,\tau'}=U^{(\mathrm{big})}_{\tau}\tilde{A}_{\tau,\tau'}(V^{(\mathrm{big})}_{\tau'})^T.
        \end{equation*}
        \item for each non-leaf node $\tau \in \mathcal{T}$ with children $\alpha$ and $\beta$ there exist $U_{\tau},V_{\tau}\in \R^{2r\times r}$ with orthonormal columns, such that
        \begin{equation} \label{eq:recU}
            U_{\tau}^{(\mathrm{big})}=\begin{bmatrix}
                U_{\alpha}^{(\mathrm{big})}&\\
                &U_{\beta}^{{(\mathrm{big})}}
            \end{bmatrix}U_{\tau} \qquad \, \mathrm{ and } \, \qquad  V_{\tau}^{{(\mathrm{big})}}=\begin{bmatrix}
                V_{\alpha}^{{(\mathrm{big})}}&\\
                &V_{\beta}^{{(\mathrm{big})}}
            \end{bmatrix}V_{\tau}.
        \end{equation} 
    \end{enumerate}
    \end{definition}

    \begin{remark} \label{remark:HSSranks}
    Point 1 of Definition~\ref{def:HSS} implies that every sub/supdiagonal block $A_{\tau,\tau'}$ has rank at most $r$. To optimize storage, one could allow for different values of $r$ for different $\tau,\tau'$. To simplify the description, we will work with a constant rank bound $r$ in this paper. On the other hand, our software implementation allows for non-constant ranks.
    \end{remark}

    Definition~\ref{def:HSS} corresponds to the usual way of storing an HSS matrix $A$~\cite{xia2010fast,massei2020hm}. For each leaf node $\tau$, the matrices $U_{\tau}:=U_{\tau}^{\mathrm{(big)}}$, $V_{\tau}:=V_{\tau}^{\mathrm{(big)}}$ and for each non-leaf node $\tau$ the matrices $U_{\tau}$, $V_{\tau}$ from~\eqref{eq:recU} are stored.  The latter pair of matrices are usually called \emph{translation operators}. The
    nestedness relation~\eqref{eq:recU} allows us to recursively recover $U_{\alpha}^{\mathrm{(big)}}$ and $V_{\alpha}^{\mathrm{(big)}}$ for any $\alpha \in \mathcal T$ from the (small) matrices $U_{\tau}$, $V_{\tau}$. We only need to additionally store the $r\times r$ matrices $\tilde A_{\tau,\tau'}$ for every pair of sibling nodes $\tau,\tau'$ and the small diagonal blocks $A_{\tau,\tau}$ (of size at most $t$) for every leaf $\tau$ in order to recover the whole matrix $A$. To see this, let $\alpha$ and $\beta$ denote the children of the root $\gamma$. Then 
    \begin{equation} \label{eqn:hssrecursion}
        A=A_{\gamma,\gamma}=\begin{bmatrix}
            A_{\alpha,\alpha} & U^{(\mathrm{big})}_{\alpha}\tilde{A}_{\alpha,\beta}(V^{(\mathrm{big})}_{\beta})^T\\
            U^{(\mathrm{big})}_{\beta}\tilde{A}_{\beta,\alpha}(V^{(\mathrm{big})}_{\alpha})^T,&A_{\beta,\beta}
        \end{bmatrix}.
    \end{equation}
    If $\alpha$ and $\beta$ are non-leaf nodes, the matrices $A_{\alpha,\alpha}$ and  $A_{\beta,\beta}$ can be recursively recovered in the same fashion. Otherwise, if $\alpha$ and $\beta$ are leaves, these matrices are stored explicitly. In summary, the matrices 
    \begin{equation}\label{eqn:data-sparse}
    \{U_{\tau},V_{\tau}: \tau \in \mathcal T\}, \quad \{\tilde A_{\tau,\tau'}:\tau,\tau' \, \mathrm{sibling}\, \mathrm{nodes}\},  \quad  \{A_{\alpha,\alpha}:\alpha \, \mathrm{leaf} \,\mathrm{node}\}
    \end{equation}
    define a data-sparse representation of an HSS matrix $A$.
    
\begin{remark} \label{remark:nonorthonormal}
     There is no need to impose orthonormality on $U_{\tau},V_{\tau}$ in the representation~\eqref{eqn:data-sparse}. The orthogonality properties required by Definition~\ref{def:HSS}.2 can always be ensured by the orthogonalization procedure described in~\cite[Section 4.2]{martinsson2011fast}, without changing the matrix $A$ represented by~\eqref{eqn:data-sparse} through the recursion~\eqref{eqn:hssrecursion}.
\end{remark}

\subsection{Block diagonal matrices and depth reduction} \label{sec:productHSS}

        To simplify notation, we make use of the following form of block diagonal matrices. Let $\{C_{\tau}\}_{\tau\in \mathcal{T}}$ be a set of matrices for a cluster tree $\T$ . Then
    \begin{equation}\label{eqn:def-blkdiag}
        \vec{C}^{(\ell)}:=\blkdiag(C_\tau: \, \tau \in \T,  \mathrm{depth}(\tau)=\ell)
    \end{equation} 
    denotes the block diagonal matrix with the diagonal blocks consisting of all matrices $C_\tau$ for which $\tau$ has depth equal to $\ell$.
    
    For example, given the data-sparse  representation~\eqref{eqn:data-sparse} of an HSS matrix $A$, the matrices $\vec{U}^{(L)}$ and $\vec{V}^{(L)}$ are block diagonal matrices containing the orthonormal matrices $U_{\alpha}$ and $V_{\alpha}$, respectively, for every leaf $\alpha \in \mathcal T$ as diagonal blocks.
When considering the product
    \begin{equation*}
        \hat A := \big( \vec{U}^{(L)} \big)^T A \vec{V}^{(L)}.
    \end{equation*}
    the representation~\eqref{eqn:data-sparse} is affected as follows:
\begin{equation} \label{eq:Ahatrepresentation}
\begin{array}{ll}
\{U_{\tau},V_{\tau}:\mathrm{depth}(\tau)\le L-1\}, \quad & 
\{U_{\alpha}^T
     U_{\alpha},V_{\alpha}^T
     V_{\alpha}: \alpha \, \mathrm{leaf} \,\mathrm{node}\}, \\
\{\tilde A_{\tau,\tau'}:\tau,\tau' \, \mathrm{sibling}\, \mathrm{nodes}\}, \quad  & \{A_{\alpha,\alpha}:\alpha \, \mathrm{leaf} \,\mathrm{node}\}.
\end{array}
\end{equation}
This representation allows us to reconstruct the matrix $\hat A$ by a recursion analogous to~\eqref{eqn:hssrecursion}. Since $U_\alpha^T U_{\alpha}=I$ and $V_\alpha^T V_{\alpha}=I$, the orthogonality  properties of Definition~\ref{def:HSS} are satisfied as well. However,
the size of $\hat A$ is $2^L m$, which is generally different from $n$ and requires the cluster tree $\mathcal T$ to be adjusted accordingly. 
\begin{definition}\label{def:balanced-tree}
Given integers $m$ and $L$, consider the index vector $[1,2,\dots, 2^L m]$. Then 
$\T_m^{(L)}$ denotes the corresponding \emph{balanced} cluster tree of depth $L$, having leaves associated with
$[(i-1)m+1,\dots,im ]$ for $i=1,\dots,2^L$.
\end{definition}
In particular, the matrix $\hat A$ defined above is an HSS matrix associated with the cluster tree $\T_m^{(L)}$. When dropping the leaves, one obtains the balanced cluster tree $\T_{2m}^{(L-1)}$ of depth $L-1$. The matrix $\hat A$ is still an HSS matrix associated with $\T_{2m}^{(L-1)}$ but its parametrization~\eqref{eq:Ahatrepresentation}

reduces to 
\[
    \{U_{\tau},V_{\tau}: \tau \in \T_{2m}^{(L-1)}\}, \quad \{\tilde A_{\tau,\tau'}:\tau,\tau' \, \mathrm{siblings}\, \mathrm{in} \, \T_{2m}^{(L-1)}\},  \quad  \{A_{\tau,\tau}:\alpha \, \mathrm{leaf} \,\mathrm{in} \, \T_{2m}^{(L-1)}\},
\]
which matches the form of~\eqref{eqn:data-sparse}.

        \section{Telescopic decompositions} \label{sec:telescopic-decomposition}
        
To develop a randomized algorithm for recovering an HSS matrix from matrix-vector products, Levitt and Martinsson~\cite{levitt2022linear} replaced the classical data-sparse representation~\eqref{eqn:data-sparse} with a certain type of telescopic decomposition. As we will see in Section~\ref{sec:from-hss-to-telescop}, the data-sparse representation~\eqref{eqn:data-sparse} is directly linked to a similar but different type of telescopic decomposition. In Section~\ref{sec:general-talescop}, we introduce a general class of telescopic decompositions that includes both types. We also discuss how to convert between these types of telescopic decompositions. 

        \subsection{A general telescopic decomposition}\label{sec:general-talescop}

         We start with a generalization of the construction from~\cite{levitt2022linear}, recursively defining a telescopic decomposition.
        \begin{definition}\label{def:telesc-decomp}
            Let $\mathcal{T}$ be a cluster tree of depth $L$ for the indices $[1,\dots,n]$. A matrix
            $A\in \R^{n\times n}$ is in \emph{telescopic decomposition} of telescopic rank $r$ if there are real matrices 
            \begin{equation*}
                \{U_{\tau},V_{\tau}:\tau\in \T, \, 1\le \mathrm{depth}(\tau) \}\quad \mathrm{and} \quad \{D_{\tau}:\tau\in \T\}
            \end{equation*}
            for brevity denoted by $\{U_{\tau},V_{\tau}, D_{\tau}\}_{\tau \in \T}$ or simply $\{U_{\tau},V_{\tau}, D_{\tau}\}$, with the following properties:
            \begin{enumerate}
            \item $D_\tau$ is of size $|\tau|\times |\tau|$ if $\mathrm{depth}(\tau)=L$ and $2r \times 2r$ otherwise;
            \item $U_\tau, V_\tau$ have orthonormal columns; they are of size $|\tau|\times r$ if $\mathrm{depth}(\tau)=L$ and $2r \times r$ otherwise;
            \item if $L=0$ (i.e., $\T$ consists only of the root $\gamma$) then $A = D_{\gamma}$;
\item if $L\ge 1$ then 
            \begin{equation}\label{eqn:telescop}
                A = \vec D^{(L)}+\vec U^{(L)}A^{(L-1)} (\vec V^{(L)})^T,
            \end{equation}
            where $\vec U^{(L)},\vec V^{(L)}$, $\vec D^{(L)}$ are the block diagonal matrices defined by $U_{\tau},V_{\tau}, D_\tau$
            as in~\eqref{eqn:def-blkdiag}, and the matrix 
            \begin{equation}\label{eqn:hatA}
                A^{(L-1)} := (\vec U^{(L)})^T(A-\vec D^{(L)}) \vec V^{(L)}
            \end{equation}
            has the telescopic decomposition $\{U_{\tau},V_{\tau}, D_{\tau}\}_{\tau \in \T_{2r}^{(L-1)}}$, where $\T_{2r}^{(L-1)}$ denotes a balanced cluster tree of depth $L-1$ (see Definition~\ref{def:balanced-tree}).
        \end{enumerate}
\end{definition}

As we will see in the following, Definition~\ref{def:telesc-decomp} offers significant freedom in the choice of diagonal blocks $D_{\tau}$, giving rise to different types of telescopic decompositions. 
A simple procedure to explicitly reconstruct the matrix $A$ from a telescopic decomposition is described in Algorithm~\ref{Alg:tel-to-full}.

\begin{algorithm}
    \begin{algorithmic}
        \Require{$\{U_{\tau},V_{\tau},D_{\tau}\}_{\tau\in \T}$ telescopic decomposition of $A$ for cluster tree $\T$ of depth $L$, }
        \Ensure{ $A$}
        \State $A\gets D_{\gamma}$ where $\gamma$ is the root of $\T$
    \For{$\ell = 1, \dots, L$} 
    \State $A\gets \vec D^{(\ell)}+\vec U^{(\ell)}A (\vec V^{(\ell)})^T$ \Comment{with $\vec D^{(\ell)},\vec U^{(\ell)}$, $\vec V^{(\ell)}$ defined as in~\eqref{eqn:def-blkdiag}}
    \EndFor  
\end{algorithmic}     

    \caption{Recovering $A$ from a telescopic decomposition.} \label{Alg:tel-to-full}
\end{algorithm}

    \subsection{From HSS matrices to telescopic decompositions}\label{sec:from-hss-to-telescop}

    The following proposition provides a construction that turns any HSS matrix into a telescopic decomposition of the same rank. 
    \begin{proposition}\label{prop:hss_to_tel}
    Let
    \[
    \{U_{\tau},V_{\tau},\tilde A_{\tau,\tau'}:\tau,\tau' \, \mathrm{sibling}\, \mathrm{nodes}\} \quad \mathrm{and} \quad  \{A_{\tau,\tau}:\tau \, \mathrm{leaf} \,\mathrm{node}\}
\]
be the data-sparse representation defining an HSS matrix $A$ of HSS rank $r$ for a cluster tree $\T$ of depth $L$.
For each node $\tau \in \T$, define
        \begin{equation*}
            D_{\tau}:=\begin{cases}
                A_{\tau,\tau} &\mathrm{if}\, \tau \, \mathrm{is\, a \, leaf \, node}\\
                \begin{bmatrix}
                    0 & \tilde A_{\alpha,\beta}\\\tilde A_{\beta,\alpha}&0
                \end{bmatrix} &\mathrm{ if } \, \tau \, \mathrm{has \, children\,} \alpha \mathrm{\, and \,}\beta.
            \end{cases}
        \end{equation*}
        Then $\{U_{\tau},V_{\tau},D_{\tau}\}$ is a telescopic decomposition of $A$ of telescopic rank $r$ associated with $\mathcal T$.
    \end{proposition}
\begin{proof}
    We proceed by induction on $L$. If $L = 0$, the tree only consists of the root $\gamma$ and the statement trivially holds because of $D_{\gamma}=A$.
    
    Let us now assume that $L\ge 1$ and consider the matrices $U^{(\mathrm{big})}_{\tau}$ and $V^{(\mathrm{big})}_{\tau}$ from Definition~\ref{def:HSS}. Because of the recursion~\eqref{eq:recU}, we have 
    \begin{equation*}
        \vec U^{(L)}(\vec U^{(L)})^T\begin{bmatrix}
             U^{(\mathrm{big})}_{\alpha}\\ & U^{(\mathrm{big})}_{\beta}\end{bmatrix}=\begin{bmatrix}
                U^{(\mathrm{big})}_{\alpha}\\ & U^{(\mathrm{big})}_{\beta}\end{bmatrix},
    \end{equation*}
                and
    \begin{equation*} \vec V^{(L)}(\vec V^{(L)})^T\begin{bmatrix}
                    V^{(\mathrm{big})}_{\alpha}\\ & V^{(\mathrm{big})}_{\beta}\end{bmatrix}=\begin{bmatrix}
                       V^{(\mathrm{big})}_{\alpha}\\ & V^{(\mathrm{big})}_{\beta}\end{bmatrix},
    \end{equation*}
    where $\alpha$, $\beta$ are the children of the root $\gamma$ and $\vec U^{(L)}$, $\vec V^{(L)}$ are the block diagonal matrices  employed in Definition~\ref{def:telesc-decomp}. Combined with the HSS recursion~\eqref{eqn:hssrecursion}, this shows that
    \begin{equation} \label{eq:aux1}
        A - \begin{bmatrix}
        A_{\alpha,\alpha} & \\
       &A_{\beta,\beta}
    \end{bmatrix} = \vec U^{(L)}(\vec U^{(L)})^T\left(A-\begin{bmatrix}
        A_{\alpha,\alpha} & \\
       &A_{\beta,\beta}
    \end{bmatrix}\right)\vec V^{(L)}(\vec V^{(L)})^T.
    \end{equation}
     Noting that $A_{\alpha,\alpha}$ and $A_{\beta,\beta}$ are HSS matrices associated with trees of depth $L-1$, we can apply induction to conclude that they are both in telescopic decomposition. This means that relations of the form~\eqref{eqn:telescop}--\eqref{eqn:hatA}  hold for both matrices or, equivalently,
     \begin{equation*}
        \begin{bmatrix}
            A_{\alpha,\alpha} & \\
           &A_{\beta,\beta}
        \end{bmatrix} - \vec D^{(L)} = \vec U^{(L)}(\vec U^{(L)})^T\left(\begin{bmatrix} A_{\alpha,\alpha}\\&A_{\beta,\beta} \end{bmatrix}-\vec D^{(L)}\right)\vec V^{(L)}(\vec V^{(L)})^T.
     \end{equation*}
     Adding this equation to~\eqref{eq:aux1} gives
     \begin{equation*} \label{eq:almostfinished}
        A=\vec D^{(L)}+\vec U^{(L)} \, A^{(L-1)} \,(\vec V^{(L)})^T, \quad \text{where} \quad A^{(L-1)}=(\vec U^{(L)})^T\big(A-\vec D^{(L)}\big)\vec V^{(L)}.
     \end{equation*}
    Together with the discussion from Section~\ref{sec:productHSS}, it follows that $A^{(L-1)}$ is an HSS matrix associated with $\T_{2r}^{(L-1)}$ and defined by the data-sparse representation
\[
    \{U_{\tau},V_{\tau}: \tau \in \T_{2r}^{(L-1)}\}, \quad \{\tilde A_{\tau,\tau'}:\tau,\tau' \, \mathrm{siblings}\, \mathrm{in} \, \T_{2r}^{(L-1)}\},  \quad  \{A^{(L-1)}_{\tau,\tau}:\tau \, \mathrm{leaf} \,\mathrm{in} \, \T_{2r}^{(L-1)}\}.
\]
\emph{If} $A^{(L-1)}_{\tau,\tau}=D_{\tau}$, this completes the proof by induction: the assumptions of this theorem are satisfied for the level-$(L-1)$ HSS matrix $A^{(L-1)}$ and thus 
$\{U_{\tau},V_{\tau},D_{\tau}\}_{\text{depth}(\tau)\le L-1}$ is a telescopic decomposition of $A^{(L-1)}$. In turn, all conditions of Definition~\ref{def:telesc-decomp} are satisfied and $\{U_{\tau},V_{\tau},D_{\tau}\}$ is a telescopic decomposition for $A$.

It remains to show that $A^{(L-1)}_{\tau,\tau}=D_{\tau}$ holds for any node $\tau$ of depth $L-1$. For this purpose, let $\alpha, \beta$ denote its children, which are leaves in $\T$. Using that $D_\alpha = A_{\alpha,\alpha}$ and $D_\beta = A_{\beta,\beta}$, we indeed obtain that
    \begin{equation*}
        A^{(L-1)}_{\tau,\tau}=\begin{bmatrix}
            U_{\alpha}^T\\&U_{\beta}^T
        \end{bmatrix}\left(A_{\tau,\tau}-\begin{bmatrix}
            D_{\alpha}\\&D_{\beta}
        \end{bmatrix}\right) \begin{bmatrix}
            V_{\alpha}\\&V_{\beta}
        \end{bmatrix}=\begin{bmatrix}
           0&  U_{\alpha}^T A_{\alpha,\beta}V_{\beta}\\U_{\beta}^TA_{\beta,\alpha}V_{\alpha}&0
        \end{bmatrix}=D_{\tau}.
    \end{equation*}
    where the last equality follows from Definition~\ref{def:HSS}.
\end{proof}

The proof of Proposition~\ref{prop:hss_to_tel} shows that the matrix $ A^{(L-1)}$ generated by the telescopic decomposition 
$\{U_{\tau},V_{\tau},D_{\tau}\}_{\tau \in \T_{2r}^{(L-1)}}$ satisfies $(A^{(L-1)})_{\tau,\tau}=D_{\tau}$ for every leaf $\tau $ of 
$\T_{2r}^{(L-1)}$. Letting $\T_{2r}^{(\ell)}$ denote a balanced cluster tree of depth $\ell \le L-1$, we can apply this property recursively and obtain
        \begin{equation}\label{prop:condition2}
            (A^{(\ell)})_{\tau,\tau}=D_{\tau}\quad  \text{for every leaf $\tau $ of 
$\T_{2r}^{(\ell)}$ for all $1\le \ell\le L-1$,}
        \end{equation}
where $A^{(\ell)}$ denotes the matrix generated by the telescopic decomposition $\{U_{\tau},V_{\tau},D_{\tau}\}_{\tau\in \T_{2r}^{(\ell)}}$. It follows from Proposition~\ref{prop:hss_to_tel} and Proposition~\ref{prop:tel_to_hss} below that telescopic decompositions with this property are in a simple one-to-one correspondence with HSS matrices, which justifies the following definition.
    
    \begin{definition}\label{def:standard-tel}
        A telescopic decomposition $\{U_{\tau},V_{\tau},D_{\tau}\}_{\tau\in \T}$ of a matrix $A$ is called $\emph{standard}$ if~\eqref{prop:condition2} holds.
    \end{definition}

\begin{proposition} \label{prop:tel_to_hss}
    Let $\T$ be a cluster tree of depth $L$ and let $A$ be the matrix generated by a standard telescopic decomposition $\{U_{\tau},V_{\tau},D_{\tau}\}_{\tau\in \T}$ of telescopic rank $r$. Then, for each nonleaf node $\tau$ with children $\alpha,\beta$, there exist matrices $\tilde A_{\alpha,\beta}, \tilde A_{\beta,\alpha} \in \R^{k\times k}$ such that 
    \begin{equation} \label{eq:dtau}
        D_{\tau}=\begin{bmatrix}            0 & \tilde A_{\alpha,\beta}\\\tilde A_{\beta,\alpha}&0 
               \end{bmatrix}.
    \end{equation} 
    Moreover, $A$ is an HSS matrix of HSS rank $r$ with the data-sparse representation
\[
    \{U_{\tau},V_{\tau}: \tau \in \mathcal T\}, \quad \{\tilde A_{\tau,\tau'}:\tau,\tau' \, \mathrm{sibling}\, \mathrm{nodes}\},  \quad  \{A_{\alpha,\alpha}:\alpha \, \mathrm{leaf} \,\mathrm{node}\}.
\]
\end{proposition}
\begin{proof}
    We proceed by induction on $L$. For $L=0$, the result trivially holds. Suppose now that $L\ge 1$. According to Definition~\ref{def:telesc-decomp}, the matrix generated by the standard telescopic decomposition $\{U_{\tau},V_{\tau},D_{\tau}\}_{\tau\in \T_{2r}^{(L-1)}}$ is the matrix $A^{(L-1)} $ defined in~\eqref{eqn:hatA}. Therefore, if $\tau$ is a node of depth $L-1$ with children $\alpha$ and $\beta$, it follows from~\eqref{prop:condition2} that
    \begin{equation*}
        D_{\tau}=A^{(L-1)}_{\tau,\tau}=\begin{bmatrix}            
            0 & U_{\alpha}^T A_{\alpha,\beta}V_{\beta}\\U_{\beta}^T A_{\beta,\alpha} V_{\alpha}&0        \end{bmatrix}
            = \begin{bmatrix}            0 & \tilde A_{\alpha,\beta}\\\tilde A_{\beta,\alpha}&0 
               \end{bmatrix},
    \end{equation*}
    where we set $\tilde A_{\alpha,\beta}= U_{\alpha}^T A_{\alpha,\beta}V_{\beta}$ and $\tilde A_{\beta,\alpha}=U_{\beta}^T A_{\beta,\alpha} V_{\alpha}$. This proves~\eqref{eq:dtau}.
    
    It remains to establish the HSS property of $A$, that is, Point 1 of Definition~\ref{def:HSS} (note that Point 2 is satisfied by construction). Combining~\eqref{eq:dtau} with the telescopic relation~\eqref{eqn:telescop}, we obtain that
    \begin{equation*}
        A_{\tau,\tau}=\begin{bmatrix} D_{\alpha}\\& D_{\beta}
        \end{bmatrix}+
        \begin{bmatrix} U_{\alpha}\\& U_{\beta}
        \end{bmatrix}
        D_{\tau}
        \begin{bmatrix} V_{\alpha}\\& V_{\beta}
        \end{bmatrix} = \begin{bmatrix}
        D_{\alpha} & U_{\alpha} \tilde A_{\alpha,\beta} V_{\beta}^T \\
        U_{\beta} \tilde A_{\beta,\alpha} V_{\alpha}^T & D_{\beta}
        \end{bmatrix}.
    \end{equation*}
    In particular, $A_{\alpha,\beta}=U_{\alpha}\tilde {A}_{\alpha,\beta}V_{\alpha}^T$, which establishes Point 1 of Definition~\ref{def:HSS} for two sibling leaves $\tau = \alpha$, $\tau^\prime = \beta$. 
    To show the corresponding property for two siblings $\tau$ and $\tau'$ of depth $\ell<L$, let  $\alpha_1,\dots,\alpha_{2^{L-\ell}}$ and $\alpha'_1,\dots,\alpha'_{2^{L-\ell}}$ denote the leaf nodes in the corresponding subtrees. 
    Using~\eqref{eqn:telescop}, the proof is completed by noting that
    \begin{eqnarray*}
        A_{\tau,\tau'}&=&
        \blkdiag(U_{\alpha_i})A^{(L-1)}_{\tau,\tau'} \blkdiag(V_{\alpha'_i})^T \\
        &=& \blkdiag(U_{\alpha_i}) \blkdiag(U_{\alpha_i})^T U^{(\mathrm{big})}_{\tau}\tilde{A}_{\tau,\tau'} \big( V^{(\mathrm{big})}_{\tau'}\big)^T \blkdiag(V_{\alpha'_i}) \blkdiag(V_{\alpha'_i})^T \\
        &=& U^{(\mathrm{big})}_{\tau}\tilde{A}_{\tau,\tau'}(V^{(\mathrm{big})}_{\tau'})^T,
    \end{eqnarray*}
    where the second equality uses 
 induction: $A^{(L-1)}$ is an HSS matrix and satisfies a relation of the form~\eqref{eqn:hssrecursion} for the parent $\gamma$ of $\tau,\tau'$.
\end{proof}

\subsection{Converting a general telescopic decomposition into a standard one}\label{sec:convert_standard}

In the following, we describe a procedure that turns an arbitrary telescopic decomposition $\{U_{\tau},V_{\tau},D_{\tau}\}$ of a matrix $A$ into a standard telescopic decomposition $\{U_{\tau},V_{\tau},C_{\tau}\}$. By the results of Section~\ref{sec:from-hss-to-telescop}, this implies the equivalence between HSS matrices of HSS rank $r$ and matrices that admit a telescopic decomposition of telescopic rank $r$. 
    For this purpose, it is crucial to understand how we can recover the principal submatrices $A_{\alpha,\alpha}$ for leaf nodes $\alpha$, since these matrices correspond to the matrices $C_{\alpha}$ in the standard telescopic decomposition. 

    \begin{proposition}\label{prop:princ-sumbatrices}
    For a cluster tree $\T$ of depth $L$, let $A$ be a matrix in telescopic decomposition $\{U_{\tau},V_{\tau},D_{\tau}\}$ of telescopic rank $r$. Then the following holds:
        \begin{enumerate}
            \item if $L=0$ (i.e., $\T$ consists only of the root $\gamma$) then $A_{\gamma,\gamma}=D_{\gamma}$;
            \item if $L\ge 1$, any pair of sibling leaf nodes $\alpha,\beta$ with parent $\tau$ satisfies
            \begin{equation}\label{eqn:ahat-in-prop}
                \begin{bmatrix}A_{\alpha,\alpha}\\&A_{\beta,\beta}\end{bmatrix}=\begin{bmatrix}D_{\alpha}\\&D_{\beta}\end{bmatrix}+\begin{bmatrix}U_{\alpha}\big[(A^{(L-1)}_{\tau,\tau})_{\mathsf 1,\mathsf 1}\big]V_{\alpha}^T\\&U_{\beta}\big[(A^{(L-1)}_{\tau,\tau})_{\mathsf 2,\mathsf 2}\big]V_{\beta}^T\end{bmatrix},
            \end{equation}
            with the matrix $A^{(L-1)}$ from Definition~\ref{def:telesc-decomp}, and $(A^{(L-1)}_{\tau,\tau})_{\mathsf 1,\mathsf 1}$ and $(A^{(L-1)}_{\tau,\tau})_{\mathsf 2,\mathsf 2}$ denoting the (1,1) and (2,2) diagonal blocks of $A^{(L-1)}_{\tau,\tau}\in \R^{2r\times 2r}$, respectively.
        \end{enumerate}

    \end{proposition}
    \begin{proof}
     Point 1 follows directly from the definition of $D_{\gamma}$. To prove Point 2, we observe that~\eqref{eqn:telescop} implies
     \begin{equation*}
        \begin{bmatrix}
            A_{\alpha,\alpha} & * \\ * &A_{\beta,\beta}
        \end{bmatrix}=A_{\tau,\tau}=\begin{bmatrix}
            D_{\alpha} &  \\&D_{\beta}
        \end{bmatrix}+\begin{bmatrix}
            U_{\alpha} &  \\&U_{\beta}
        \end{bmatrix}A^{(L-1)} _{\tau,\tau}\begin{bmatrix}
            V^T_{\alpha} &  \\&V^T_{\beta}
        \end{bmatrix}.
     \end{equation*}
     Therefore, taking the diagonal blocks concludes the proof.
    \end{proof}

    The previous proposition combined with the fact that a telescopic decomposition of the matrix $A^{(L-1)}$ employed in \eqref{eqn:ahat-in-prop} is given by  $\{U_{\tau},V_{\tau},D_{\tau}\}_{\mathrm{depth}(\tau)\le L-1}$ (see Definiton~\ref{def:telesc-decomp}), results in a practical way to compute the matrices $A_{\alpha,\alpha}$ for all leaves $\alpha$; see Algorithm~\ref{Alg:principal-submat}.

    \begin{algorithm}
        \begin{algorithmic}
            \Require{Matrix $A$  in telescopic decomposition $\{U_{\tau},V_{\tau},D_{\tau}\}$ for cluster tree $\T$ of depth $L$}
            \Ensure{ \{$A_{\alpha,\alpha}:\alpha$ leaf node\}}
        \State $\hat A_{\gamma,\gamma}\gets D_{\gamma}$ for root $\gamma$ of $\T$
        \For{$\ell = 0,\dots, L-1$}
        \For {each node $\tau$ of depth $\ell$}
        \State Denoting by $\alpha,\beta$ the children of $\tau$ and defining $(\hat A_{\tau,\tau})_{\mathsf 1,\mathsf 1}$, $(\hat A_{\tau,\tau})_{\mathsf 2,\mathsf 2}$ as in Proposition~\ref{prop:princ-sumbatrices}
        \State $\left[\begin{smallmatrix}\hat A_{\alpha,\alpha}\\&\hat A_{\beta,\beta}\end{smallmatrix}\right]\gets \left[\begin{smallmatrix}D_{\alpha}\\&D_{\beta}\end{smallmatrix}\right]+\left[\begin{smallmatrix}U_{\alpha}\big[(\hat A_{\tau,\tau})_{\mathsf 1,\mathsf 1}\big]V_{\alpha}^T\\&U_{\beta}\big[(\hat A_{\tau,\tau})_{\mathsf 2,\mathsf 2}\big]V_{\beta}^T\end{smallmatrix}\right]$
        \EndFor
        \EndFor
        \For {each leaf node $\alpha$}
        \State $A_{\alpha,\alpha} \gets \hat A_{\alpha,\alpha}$
        \EndFor
    \end{algorithmic}
    
        \caption{Computation of principal submatrices of $A$ given in telescopic decomposition} \label{Alg:principal-submat}
    
    \end{algorithm}

    To satisfy condition~\eqref{prop:condition2} of a standard telescopic decomposition on the leaf level, we need to set 
    \begin{equation*}
        C_{\alpha}:=A_{\alpha,\alpha}
    \end{equation*}
    for each leaf node $\alpha$.
    Moreover, if $L\ge 1$, for each node $\tau$ of depth $L-1$ the matrix $C_{\tau}$ is given by $A^{(L-1)}_{\tau,\tau}$, where $A^{(L-1)}$ is now defined as    \begin{equation*}
        A^{(L-1)} := (\vec U^{(L)})^T\, \big(A-\vec C^{(L)}\big) \, \vec V^{(L)}=(\vec U^{(L)})^T\, \big(A-\vec D^{(L)}\big) \, \vec V^{(L)}+(\vec U^{(L)})^T\, \big(\vec D^{(L)}-\vec C^{(L)}\big) \, \vec V^{(L)},
    \end{equation*}
    with the block diagonal matrices $\vec U^{(L)},\vec V^{(L)},\vec D^{(L)}$ and $\vec C^{(L)}$ defined from $\{U_{\tau}\},  \{V_{\tau}\},\{D_{\tau}\}$ and $\{C_{\tau}\}$, respectively, according to \eqref{eqn:def-blkdiag}. By Definition~\ref{def:telesc-decomp}, the matrix $(\vec U^{(L)})^T\, \big(A-\vec D^{(L)}\big) \, \vec V^{(L)}$ is generated by the telescopic decomposition $\{U_{\tau},V_{\tau},D_{\tau}\}_{\mathrm{depth}(\tau)\le L-1}$, hence defining 
    \begin{equation*}
        \hat D_\tau :=\begin{cases}
             D_{\tau}-\begin{bmatrix}
                U_{\alpha}^T(D_{\alpha}-C_{\alpha})V_{\alpha}\\& U_{\beta}^T(D_{\beta}-C_{\beta})V_{\beta} 
            \end{bmatrix}\quad &\, \mathrm{if}\, \tau \, \mathrm{has\, depth} \, L-1 \, \mathrm{and \, children}\, \alpha,\beta;\\
            D_{\tau} &\mathrm{otherwise;}
        \end{cases}
    \end{equation*}
    the matrix $A^{(L-1)}$ is generated by the telescopic decomposition $\{U_{\tau},V_{\tau},\hat D_{\tau}\}_{\mathrm{depth}(\tau)\le L-1}$. Therefore a standard decomposition of $A$ can be computed by iterating Algorithm~\ref{Alg:principal-submat}, as summarized in Algorithm~\ref{Alg:telesc-to-standard}. In particular, the computational complexity of transforming a telescopic decomposition into a standard one is $\mathcal{O}(r^3 2^L)$ where $r$ is the telescopic rank of $A$. Assuming the threshold size and the telescopic rank to be constant, this shows linear complexity in the size of $A$.

    \begin{algorithm}
        \begin{algorithmic}
            \Require{Matrix $A$  in telescopic decomposition $\{U_{\tau},V_{\tau},D_{\tau}\}$ for cluster tree $\T$ of depth $L$}
            \Ensure{Standard telescopic decomposition $\{U_{\tau},V_{\tau},C_{\tau}\}$ of $A$}
        \State \{$C_{\alpha}:\alpha$ leaf node\} $\gets $ 
        Algorithm~\ref{Alg:principal-submat} applied to $\{U_{\tau},V_{\tau},D_{\tau}\}$
        \For {each node $\tau$ }
        \State $\hat D_{\tau} \gets D_{\tau}$
        \EndFor
        \For{$\ell = L-1,\dots, 0$}
        \For {each node $\tau$ of depth $\ell$}
        \State Denoting by $\alpha,\beta$ the children of $\tau$
        \State $\hat D_\tau \gets
            \hat D_{\tau}-\begin{bmatrix}
               U_{\alpha}^T(\hat D_{\alpha}-C_{\alpha})V_{\alpha}\\& U_{\beta}^T(\hat D_{\beta}-C_{\beta})V_{\beta} 
           \end{bmatrix}$
        \EndFor
        \State \{$C_{\tau}:\tau \in \T$ of depth $\ell$\} $\gets $ 
        Algorithm~\ref{Alg:principal-submat} applied to $\{U_{\tau},V_{\tau},\hat D_{\tau}\}_{\mathrm{depth}(\tau)\le l}$
        \EndFor
    \end{algorithmic}
    
        \caption{Computation of a standard telescopic decomposition of $A$ given in telescopic factors.} \label{Alg:telesc-to-standard}
    
    \end{algorithm}
    
\subsection{Symmetric telescopic decompositions} \label{sec:symmetric}

For a symmetric matrix, the definition of a telescopic decomposition can be adjusted to reflect symmetry. 

    \begin{definition}\label{def:symm-tel-dec}
        A telescopic decomposition $\{U_{\tau}, V_{\tau}, D_{\tau}\}_{\tau \in \T}$ is said to be \emph{symmetric} if $V_{\tau} =U_{\tau}$ and $D_{\tau}^T = D_{\tau}$ hold for every $\tau \in \T$. In analogy to the nonsymmetric case, we employ the term \emph{standard} if \eqref{prop:condition2} is satisfied.
        For simplicity, a symmetric telescopic decomposition is denoted by $\{U_{\tau},D_{\tau}\}_{\tau \in \T}$, avoiding the repetition of $U_{\tau}$.
    \end{definition}

    If $A$ is symmetric and has HSS rank $r$, there exist data-sparse representations of the form~\eqref{eqn:data-sparse}, for which $U_{\tau}=V_{\tau}$ have $r$ columns and $\tilde A_{\tau,\tau'}^T=\tilde A_{\tau',\tau}$ holds for each pair of sibling nodes $\tau,\tau'$ (see \cite[Section~4.1]{martinsson2011fast}). Therefore, Proposition~\ref{prop:hss_to_tel} implies that $A$ admits a symmetric telescopic decomposition of rank $r$. We also recall that the procedure described in Section~\ref{sec:convert_standard} converts a telescopic decomposition $\{U_{\tau},V_{\tau},D_{\tau}\}$ into a standard one, without changing the matrices $U_{\tau}, V_{\tau}$. In turn, the same procedure can be employed to convert a symmetric telescopic decomposition into a standard symmetric telescopic decomposition.

\section{Computing telescopic decompositions for functions of symmetric HSS matrices}\label{sec:comp-fun}

    If $A$ is a symmetric HSS matrix
    with spectrum contained in $[\lambda_{\min}, \lambda_{\max}]$ 
    and the function $f$ is analytic
    on $[\lambda_{\min}, \lambda_{\max}]$, then $f(A)$ can usually be well 
    approximated by an HSS matrix. While this has been observed before~\cite[Section~3]{cortinovis2022divide}, it is nontrivial to develop an algorithm that fully exploits this property. In this section, we derive such an algorithm that computes a telescopic decomposition for an HSS approximation of $f(A)$ starting from a standard symmetric
    telescopic decomposition of $A$. If needed, this can be converted into a standard telescopic decomposition, employing the results of Section~\ref{sec:convert_standard}, and therefore into an HSS data-sparse representation~\eqref{eqn:data-sparse}.
    If $f$ is a rational function of a certain degree, we show that such an approximation is exact and, otherwise, the approximation error is bounded using a rational approximation of $f$ on $[\lambda_{\min}, \lambda_{\max}]$. 

    The following theorem summarizes results from~\cite{beckermann2021low} fundamental for our purpose. It shows how the low-rank update of a matrix function can be approximated using the rational Krylov subspaces from Section~\ref{sec:Krylov}.
    \begin{theorem}[{\cite[Theorem~4.5]{beckermann2021low}}] \label{thm:BCKS}
        Let $A = D + Z \tilde A Z^T,$ where $D\in \R^{n\times n}, \tilde A \in \R^{m\times m}$ are symmetric and $Z\in\R^{n\times m}$. Let $[\lambda_{\min}, \lambda_{\max}]$ be an interval that contains the spectra of $A$ and $D$, and let $\vec{\xi}_k = ( \xi_0,\dots,\xi_{k-1})$, with $\xi_j \in \C\cup\{\infty\}$, be a list of poles closed under complex conjugation. For a function $f$ analytic on 
        $[\lambda_{\min}, \lambda_{\max}]$, we consider the  approximation
        \begin{equation}\label{eqn:approxBCKS}
            f(A) \approx \hat F:= f(D) + W \big(f(W^TAW)-f(W^TDW)\big) W^T,
        \end{equation}
        where $W$ is an orthonormal basis of $\rat(D,Z, \vec{\xi}_k)$. Then the approximation error $E(f):= \hat F - f(A)$ satisfies
        \begin{equation*}
            \norm{E(f)}_2\le 4 \min_{r\in \pol_{k}/q_k}\norm{f-r}_{\infty},
        \end{equation*}
        where $q_k =\prod_{\xi\in\vec\xi_k, \xi\neq \infty}(x-\xi)$, $\pol_k$ denotes the set of polynomials of degree at most $k-1$, and $\norm{\cdot}_{\infty}$ denotes the supremum norm on $[\lambda_{\min}, \lambda_{\max}]$. In particular, the approximation \eqref{eqn:approxBCKS} is exact if $f\in \pol_{k}/q_k$.
    \end{theorem}
 
    We will apply Theorem~\ref{thm:BCKS} recursively to telescopic decompositions. In order to do so conveniently, we slightly loosen our assumptions on a standard telescopic decomposition. Given a cluster tree $\T$ associated with $[1,\dots,n]$ we assume that a matrix $A\in \R^{n\times n}$ can be written as 
    \begin{equation} \label{eq:telescopic-nonorthogonal}
        A = \vec{\tilde D}^{(L)} + \vec Z^{(L)} {A}^{(L-1)} (\vec Z^{(L)})^T, 
    \end{equation}
    where:
    \begin{itemize}
        \item ${A}^{(L-1)}$ admits a standard symmetric 
          telescopic decomposition $\{ U_\tau, D_\tau \}_{\tau \in \T_{2r}^{(L-1)}}$;
        \item $\vec {\tilde D}^{(L)} = \blkdiag( \tilde D_\tau \ : \ \tau\ 
          \mathrm{leafnode\ in}\ \T)$, and $\tilde D_\tau = A_{\tau \tau}$;
        \item $\vec Z^{(L)} = \blkdiag( Z_\tau \ : \ \tau\ 
          \mathrm{leafnode\ in}\ \T)$, with $Z_\tau \in \mathbb R^{|\tau| \times r}$.
    \end{itemize}
    The key difference to assuming that $A$ has a standard symmetric 
    telescopic decomposition is that no orthogonality is enforced on
    $\vec Z^{(L)}$, the factors on the leaf level. 
    The key advantage of~\eqref{eq:telescopic-nonorthogonal} is that it remains unaffected when multiplying with certain block diagonal matrices.
    The following proposition additionally shows how to move one level up.

  \begin{proposition} \label{prop:nonorthogonal-telescopic}
      Let $A$ be an $n \times n$ matrix admitting the decomposition~\eqref{eq:telescopic-nonorthogonal}, and 
    $\vec{W}^{(L)} = \blkdiag(W_{\tau} \ : \ \tau\ \mathrm{leaf \ node})$ with 
    $W_{\tau}\in \R^{|{\tau}|\times m}$. Then, 
    \[
      (\vec{W}^{(L)})^T A \vec{W}^{(L)} = 
        \tilde{\vec D}^{(L-1)} + 
        \vec Z^{(L-1)} {A}^{(L-2)} (\vec Z^{(L-1)})^T, 
    \]
    with the block diagonal matrices
    \begin{equation*}
        \tilde{\vec{D}}^{(L-1)} = \blkdiag(\tilde D_{\tau}:\mathrm{depth}(\tau)=L-1), \quad {\vec{Z}}^{(L-1)} = \blkdiag(Z_{\tau}:\mathrm{depth}(\tau)=L-1)
    \end{equation*}
    containing the diagonal blocks
    \begin{align*}
        Z_\tau &= \begin{bmatrix}
            W_{\alpha}^T Z_{\beta} \\
            & W_{\alpha}^T Z_{\beta} \\ 
        \end{bmatrix} U_\tau \in \R^{2m\times r}, \\
        \tilde D_\tau &= \begin{bmatrix}
            W_{\alpha}^T \tilde D_{\alpha} W_{\alpha} \\ 
            & W_{\beta}^T \tilde D_{\beta} W_{\beta} 
        \end{bmatrix} + \begin{bmatrix}
            W_{\alpha}^T Z_{\alpha} \\
            & W_{\beta}^T Z_{\beta} \\ 
        \end{bmatrix} D_\tau \begin{bmatrix}
            W_{\alpha}^T Z_{\alpha} \\
            & W_{\beta}^T Z_{\beta} \\ 
        \end{bmatrix}^T \in \R^{2m\times 2m},
    \end{align*}
    and the matrix ${\vec{A}}^{(L-2)}$ generated by the standard telescopic decomposition $\{U_{\tau}, D_{\tau}\}_{\tau \in \T_{2r}^{(L-2)}}$, where the matrices $U_{\tau}$ and $D_{\tau}$ stem from the telescopic decomposition of $A^{(L-1)}$. Moreover, $\tilde D_{\tau} = A_{\tau,\tau}$ holds for each leaf node $\tau\in \T^{(L-1)}_{2m}$.
  \end{proposition}
  \begin{proof}
    Considering \eqref{eq:telescopic-nonorthogonal} and applying the telescopic decomposition~\eqref{eqn:telescop} of the matrix $A^{(L-1)}$ we get
    \begin{align*}
        (\vec{W}^{(L)})^T A \vec{W}^{(L)} &= (\vec{W}^{(L)})^T\left(\vec{\tilde D}^{(L)}  + \vec Z^{(L)} (\vec D^{(L-1)}+\vec U^{(L-1)} A^{(L-2)} (\vec U^{(L-1)}))^T(\vec Z^{(L)})^T\right)\vec{W}^{(L)} \\
        &= \vec{\tilde D}^{(L-1)} + 
        \vec Z^{(L-1)} {A}^{(L-2)} (\vec Z^{(L-1)})^T
    \end{align*}
    with $\vec D^{(L-1)},\vec U^{(L-1)}$ defined from $\{D_{\tau}\}$, $\{U_{\tau}\}$, in accordance with \eqref{eqn:def-blkdiag}. Moreover, because $\{U_{\tau},D_{\tau}\}_{\tau\in\T_{2r}^{(L-1)}}$ is a standard decomposition of $A^{(L-1)}$, the relation $\tilde D_{\tau} = A_{\tau,\tau}$ holds for each leaf node $\tau\in \T^{(L-1)}_{2m}$.   \end{proof}

To approximate $f(A)$ for a matrix $A$ admitting the decomposition~\eqref{eq:telescopic-nonorthogonal}, we use the construction of Theorem~\ref{thm:BCKS}
    with $D = \tilde{\vec{D}}^{(L)}$ and 
    $Z = \vec Z^{(L)}$, which yields the approximation
    \begin{equation} \label{eq:frecursive}
        f(A) \approx f(\tilde{\vec{D}}^{(L)}) + 
          \vec{W}^{(L)} \left[
            f\big((\vec{W}^{(L)})^T A \vec{W}^{(L)}\big) - f\big((\vec{W}^{(L)})^T \tilde{\vec{D}}^{(L)} \vec{W}^{(L)}\big)
          \right] (\vec{W}^{(L)})^T, 
    \end{equation}
    where $\vec{W}^{(L)}$ is an orthogonal basis for 
    $\rat(\tilde{\vec{D}}^{(L)},\vec Z^{(L)}, \vec{\xi}_k)$. 
    Note that the three evaluations of $f$ are all well defined because $\tilde{\vec{D}}^{(L)}$ contains diagonal blocks of $A$ and 
    $(\vec{W}^{(L)})^T \tilde{\vec{D}}^{(L)} \vec{W}^{(L)}$, $(\vec{W}^{(L)})^T A \vec{W}^{(L)}$ are
    orthogonal compressions. By eigenvalue interlacing, the spectra of these three matrices 
    are contained in $[\lambda_{\min}, \lambda_{\max}]$.
    We now make 
    two observations:
    \begin{enumerate}[(i)]
        \item Proposition~\ref{prop:blkdiag_krylov} implies that the $n \times 2^L rk$ matrix $\vec{W}^{(L)}$ takes the form
           \[
                \vec{W}^{(L)} = \blkdiag(
                    W_{\tau} \ : \ 
                    \tau\ \mathrm{leaf\ of}\ \T
                ), 
                \ \mathrm{with}\ 
                W_{\tau} \in \R^{|\tau| \times rk} \text{ orthonormal basis of }
                \rat(\tilde D_\tau, Z_\tau, \vec{\xi}_k). 
            \]
        \item Proposition~\ref{prop:nonorthogonal-telescopic} implies that 
        $B^{(L-1)} := 
          (\vec{W}^{(L)})^T A \vec{W}^{(L)}$ admits the  
          decomposition
          \begin{equation} \label{eq:blm1}
            B^{(L-1)} = 
              \tilde{\vec D}^{(L-1)} + 
              \vec Z^{(L-1)} {A}^{(L-2)} (\vec Z^{(L-1)})^T. 
          \end{equation}
    \end{enumerate}

    In~\eqref{eq:frecursive}, the function $f$ needs to be evaluated for three matrices. This is cheap for 
    $\tilde{\vec{D}}^{(L)}$ and 
    $\vec{W}^{(L)})^T \tilde{\vec{D}}^{(L)} \vec{W}^{(L)}$ because these matrices are block diagonal with small diagonal blocks, for which the evaluation of $f$ is computed explicitly.
    The expensive part is the evaluation of $f$ for $B^{(L-1)} = 
          (\vec{W}^{(L)})^T A \vec{W}^{(L)}$. For this purpose, we use the decomposition~\eqref{eq:blm1} and apply the approximation~\eqref{eq:frecursive} again:
    \begin{align*}
        f(B^{(L-1)}) 
          &\approx f(
            \tilde {\vec D}^{(L-1)}) 
            + \vec W^{(L-1)} \left[
            f(B^{(L-2)}) - 
            f( 
             (\vec W^{(L-1)})^T 
             \tilde{\vec{D}}^{(L-1)}
             \vec W^{(L-1)}
             )  
          \right]
          (\vec W^{(L-1)})^T, 
    \end{align*}
    where $B^{(L-2)} := (\vec W^{(L-1)})^T B^{(L-1)} \vec W^{(L-1)}$
    and 
     \[
        \vec{W}^{(L-1)} = \blkdiag(
            W_{\tau} \ : \ 
            \tau\ \mathrm{leaf\ of}\ \T_{2rk}^{(L-1)}
        ), 
        \ \mathrm{with}\ 
        W_{\tau} \text{ orthonormal basis of }
        \rat(\tilde D_\tau,  Z_\tau, \vec{\xi}_k). 
    \]  
    Since $\tilde{\vec{D}}^{(L-1)}$ contains diagonal blocks 
    of $B^{(L-1)}$, the three evaluations of $f$ are again well defined.
    
    The procedure described above is repeated recursively until reaching a tree of depth $0$, 
    and at that point one simply computes
    the matrix function of the corresponding 
    dense matrix of size $2rk \times 2rk$ explicitly. The $i$th step of the recursive procedure proceeds as follows: one assumes a decomposition of the form
    \begin{equation}\label{eqn:Bi}
            B^{(L-i)} = 
              \tilde{\vec D}^{(L-i)} + 
              \vec Z^{(L-i)} {A}^{(L-i-1)} (\vec Z^{(L-i)})^T,
    \end{equation}
    where $\tilde{\vec D}^{(L-i)}$, $\vec Z^{(L-i)}$ are block diagonal matrices defined from $\{\tilde{D}_{\tau}\}$, $\{Z_{\tau}\}$, in accordance with~\eqref{eqn:def-blkdiag}, and ${A}^{(L-i-1)}$ admits a standard symmetric telescopic decomposition $\{U_{\tau},D_{\tau}\}_{\tau \in \T^{(L-i-1)}_{2r}}$. Then $f(B^{(L-i)})$ is approximated by 
    \begin{equation}\label{eqn:FBi}
            f(\tilde {\vec D}^{(L-i)} )
            + \vec W^{(L-i)} \left[
            f(B^{(L-i-1)}) - 
            f( 
             (\vec W^{(L-i)})^T 
             \tilde{\vec{D}}^{(L-i)}
             \vec W^{(L-i)}
             )  
          \right]
          (\vec W^{(L-i)})^T, 
    \end{equation}
    where $B^{(L-i-1)} := (\vec W^{(L-i)})^T B^{(L-i)} \vec W^{(L-i)}$
    and 
     \[
        \vec{W}^{(L-i)} = \blkdiag(
            W_{\tau} \ : \ 
            \tau\ \mathrm{leaf\ of}\ \T_{2rk}^{(L-i)}
        ), 
        \ \mathrm{with}\ 
        W_{\tau} \text{ orthonormal basis of }
        \rat(\tilde D_\tau,  Z_\tau, \vec{\xi}_k),
    \] 
    explicitly computing $f(\tilde {\vec D}^{(L-i)} )$, $ 
    f( 
     (\vec W^{(L-i)})^T 
     \tilde{\vec{D}}^{(L-i)}
     \vec W^{(L-i)}
    )$ and recursively approximating $f(B^{(L-i-1)})$. In particular, the procedure can be iterated considering the decomposition
    \begin{equation*}
        B^{(L-i-1)} = 
          \tilde{\vec D}^{(L-i-1)} + 
          \vec Z^{(L-i-1)} {A}^{(L-i-2)} (\vec Z^{(L-i-1)})^T,
    \end{equation*}
    given by Proposition~\ref{prop:nonorthogonal-telescopic}.
    
    Assuming that $A$ admits a standard (symmetric) telescopic decomposition $\{U_{\tau},D_{\tau}\}$, the described procedure starts by taking $Z_{\tau}=U_{\tau}$ and $\tilde D_{\tau}=D_{\tau}$ for each leaf node $\tau$. It results in a 
    telescopic decomposition for an 
    approximation of $f(A)$, with generators 
    $\{ W_\tau, C_{\tau} \}$; the generators $W_{\tau}$ are defined by the 
    rational Krylov subspaces constructed throughout the process, 
    whereas $C_{\tau}$ takes the form
    \begin{equation*} 
        C_{\tau}:=\begin{cases}
        f(\tilde D_{\tau}) &\,\mathrm{ if } \,\tau \, \mathrm{ is \, the \, root \, node}\\
        f(\tilde D_{\tau})-W_{\tau}f(W_{\tau}^T\tilde D_{\tau} W_{\tau}) W_{\tau}^T \quad & \, \mathrm{ otherwise.}
    \end{cases}
    \end{equation*}
    The procedure is summarized in Algorithm~\ref{Alg:main2}.
    
    Let us emphasize that even if the decomposition\eqref{eq:telescopic-nonorthogonal} is used in the course of the algorithm, 
    the final result has a symmetric telescopic decomposition in the sense of Definition~\ref{def:symm-tel-dec}. The 
    non-orthogonal factors $\vec{Z}^{(\ell)}$ are only needed to represent intermediate stages.
   
    \begin{algorithm}
        \begin{algorithmic}
            \Require{$\{U_{\tau},D_{\tau}\}$ standard symmetric telescopic decomposition of matrix $A$, function $f$, list of poles $\vec {\xi}_k= (\xi_0, \dots, \xi_{k-1})\subseteq \C\cup\{\infty\}$ closed under complex conjugation}
            \Ensure{ $\{W_{\tau},C_{\tau}\}$ telescopic factorization of an approximation to $f(A)$}
        
        \For{$\ell = L, L-1, \dots, 0$}
        \For {each node $\tau$ of depth $\ell$}
        \If {$\ell =L$}
        \State $Z_{\tau}\gets U_{\tau}$ 
        \State $\tilde D_{\tau}\gets D_{\tau}$
        \Else 
        \State Let $\alpha$ and $\beta$ be the children of $\tau$
            \State $Z_{\tau}\gets\left[\begin{smallmatrix} W_{\alpha}^T Z_{\alpha}\\ &W_{\beta}^T Z_{\beta}\end{smallmatrix}\right]U_{\tau}$
            \State $\tilde D_{\tau}\gets\left[\begin{smallmatrix}
                W_{\alpha}^T\tilde D_{\alpha}W_{\alpha}\\&W_{\beta}^T\tilde D_{\beta}W_{\beta}
            \end{smallmatrix}\right]+\left[\begin{smallmatrix} W_{\alpha}^T Z_{\alpha}\\ &W_{\beta}^T Z_{\beta}\end{smallmatrix}\right] D_{\tau} \left[\begin{smallmatrix} Z_{\alpha}^TW_{\alpha}\\ &Z_{\beta}^TW_{\beta}\end{smallmatrix}\right]$
            \EndIf
            \If{$\ell =0$}
            \State $C_{\tau}\gets f(\tilde D_{\tau})$
            \Else
            \State $W_{\tau}\gets$ orthonormal basis of $\rat(\tilde D_{\tau}, Z_{\tau},\vec {\xi}_k)$ 
           
            \State $C_{\tau}\gets f(\tilde D_{\tau})-W_{\tau}f(W_{\tau}^T\tilde D_{\tau}W_{\tau})W_{\tau}^T$
        \EndIf
        \EndFor
        \EndFor
        
    \end{algorithmic}
    
        \caption{Computation of $f(A)$ for symmetric HSS matrix $A$ in telescopic decomposition.} \label{Alg:main2}
    
    \end{algorithm}

    To discuss the complexity of Algorithm~\ref{Alg:main2}, let $r$ denote the HSS/telescopic rank of $A$, assume that $n = 2^L t$, with the threshold size $t \sim kr$, and that the cluster tree is balanced: $\T = \T_{t}^{(L)}$.
    On the leaf level $L$, the computation of all $W_\tau$ and $C_\tau$ requires $\mathcal O(2^L ( t^3 +  k r t^2  ) ) = \mathcal O( n k^2 r^2)$ operations.
    On level $\ell < L$, the complexity is $\mathcal O(2^{\ell} k^3 r^3 ) = \mathcal O( 2^{-(L-\ell)} n k^2 r^2 )$. This gives a total complexity of 
    \begin{equation} \label{eq:complexity}
     \mathcal O( n k^2 r^2),
    \end{equation}
    which is linear in $n$ \emph{if} both $r$ and $k$ are considered constant.

    We conclude this section with a result that bounds the approximation error of Algorithm~\ref{Alg:main2} by
    the rational approximation error of $f$ on $[\lambda_{\min},\lambda_{\max}]$. 

    \begin{theorem}\label{thm:error_bound}
    With the notation introduced in Theorem~\ref{thm:BCKS}, let $A$ be a symmetric HSS matrix associated with a cluster tree $\T$ of depth $L$.
    Let $f$ be analytic on an interval $[\lambda_{\min},\lambda_{\max}]$ containing the eigenvalues of $A$.
    Letting $E(f)$ denote the difference between $f(A)$ and the ouput of Algorithm~\ref{Alg:main2} applied to a standard decomposition of $A$, it holds that
        \begin{equation*}
            \norm{E(f)}_2\le 4L\min_{r\in \pol_{k}/q_k}\norm{f-r}_{\infty}.
        \end{equation*}
    \end{theorem}

    \begin{proof} Let $\{W_{\tau},C_{\tau}\}_{\tau\in\T}$ be the telescopic decomposition returned by Algorithm~\ref{Alg:main2} applied to a standard symmetric telescopic decomposition $\{U_{\tau},D_{\tau}\}_{\tau\in\T}$. For each $1\le i \le L-1$ let $B^{(L-i)}$ be the matrices defined in \eqref{eqn:Bi}, and for each leaf node $\tau\in \T^{(L-i)}_{2rk}$, let $\tilde D_{\tau}= B^{(L-i)}_{\tau,\tau}$. Moreover, to streamline the notation, we let $B^{(L)}:=A$ and $\tilde{D}_{\tau}:=D_{\tau}$ for each leaf node $\tau\in \T$. For each $i$, let $E^{(L-i)}(f)$ be the difference between $f(B^{(L-i)})$ and its approximation~\eqref{eqn:FBi}. Since $[\lambda_{\min}, \lambda_{\max}]$ contains the eigenvalues of $B^{(L-i)}$ and $\tilde D_{\tau}$ for each $\tau$, Theorem~\ref{thm:BCKS} implies 
        \begin{equation} \label{eq:intermbound}
            \norm{E^{(L-i)}(f)}_2\le 4 \min_{r\in \pol_{k}/q_k}\norm{f-r}_\infty.
        \end{equation}
        Denoting by $F^{(L-i)}$ the matrix generated by the symmetric telescopic decomposition $\{W_{\tau},C_{\tau}\}_{\tau \in \T^{L-i}_{2rk}}$ for $0\le i\le L$, we have
        \begin{equation*}
            f(B^{(L-i)})-F^{(L-i)}=
            \begin{cases}
                0 &\mathrm{ if }\, i=L;\\
                f(B^{(L-i)})-\vec C^{(L-i)}-\vec W^{(L-i)}F^{(L-i-1)}(\vec W^{(L-i-1)})^T &\mathrm{otherwise},
            \end{cases}
        \end{equation*}
         where $\vec C^{(L-i)}$, $\vec W^{(L-i)}$ are the block diagonal matrices defined by $\{C_{\tau}\}$, $\{W_{\tau}\}$, in accordance with \eqref{eqn:def-blkdiag}. For $i<L$ we observe that 
        \begin{equation*}
            f(B^{(L-i)})-\vec  C^{(L-i)}=\vec W^{(L-i)}f(B^{(L-i-1)})(\vec W^{(L-i)})^T+E^{(L-i)}(f),
        \end{equation*}
        and, hence,
        \begin{equation*}
            f(B^{(L-i)})-F^{(L-i)}=\vec W^{(L-i)}\big(f(B^{(L-i-1)})-F^{(L-i-1)}\big)(\vec W^{(L-i)})^T+E^{(L-i)}(f).
        \end{equation*}
        Using that that the matrices $\vec W^{(L-i)}$ have orthonormal columns, this implies
        \begin{equation*}
            \norm{E(f)}_2=\norm{f(B^{(L)})-F^{(L)}}_2\le \sum_{i=0}^{L-1} \norm{E^{(L-i)}(f)}_2,
        \end{equation*}
        which concludes the proof after applying the inequality~\eqref{eq:intermbound}.   
    \end{proof}

\section{Pole selection}\label{sec:pole_selection}

Theorem~\ref{thm:error_bound} shows that the choice of poles in the rational Krylov subspaces $\rat(\tilde D_{\tau}, Z_{\tau},\vec {\xi}_k)$ is critical to the convergence of Algorithm~\ref{Alg:main2}. Normally, repeated poles are preferred to reduce the cost of solving the shifted linear systems needed for constructing a basis of the subspace~\cite{guttel2013rational}. However, such considerations do not apply to Algorithm~\ref{Alg:main2}; solving linear systems with the small matrix $\tilde D_{\tau}$ (shifted by a pole) is cheap. 

By Theorem~\ref{thm:error_bound}, if $f$ is an analytic function on an interval $[a,b]$ that contains the eigenvalues of $A$ and the list of poles $\vec \xi_k$ satisfies
\begin{equation} \label{eq:errorboundpoles}
    \min_{r\in \pol_{k}/q_k}\norm{f-r}_\infty\le \frac {\epsilon}{4L}
\end{equation}
then Algorithm~\ref{Alg:main2} returns an approximation of $f(A)$ within an error bounded by a user-specified tolerance $\epsilon>0$. In the following, we describe explicit pole selection strategies that ensure~\eqref{eq:errorboundpoles} for two important classes of functions. For general $f$, general rational approximation methods, like the AAA algorithm~\cite{nakatsukasa2018aaa}, can be used to select the poles.

\subsection{Exponential function}

In the context of the matrix exponential, it is not uncommon to use polynomial approximations, that is, all poles are infinite. However, the corresponding (polynomial) Krylov subspace methods often converge poorly when the spectrum is wide, that is, $a \ll b$; see~\cite{beckermann2009error,Hochbruck1997} for theoretical results. As their computational overhead is small in our setting, it is preferable to use rational approximations/Krylov subspaces. Assuming $b\le 0$ (which can always be attained by shifting the matrix), it is well known~\cite{Gonchar1987}
that for every $k$ there exists a list of poles $\vec \xi_k$ and a (universal) constant $C$ such that
\begin{equation*} 
    \min_{r\in \pol_{k}/q_k}\norm{f-r}_\infty \le C K_e^{-k},\quad  K_e \approx 9.289.
\end{equation*}
In turn, this means that it suffices to choose
\begin{equation*}
    k\ge  \log\big( 4LC \epsilon^{-1} \big) \big/\log(K_e)
\end{equation*}
such that Algorithm~\ref{Alg:main2} applied to a symmetric negative semi-definite HSS matrix $A$ with non-positive eigenvalues returns an approximation with an error $\epsilon$.
In particular, note that these estimates are independent of the width of the spectrum. Following~\eqref{eq:complexity}, this gives
a complexity of \[ \mathcal O( n (\log \log n + \log \epsilon^{-1} )^2 r^2),\] where $r$ is the telescopic/HSS rank of $A$. For example, this implies that the fixed-accuracy approximation to the exponential of \emph{any} tridiagonal symmetric negative semi-definite matrix $A$ has nearly linear complexity $\mathcal O( n ( \log \log n )^2 )$. We are not aware of any other algorithm that can achieve this. 

\subsection{Markov functions}

 Following the exposition in~\cite{beckermann2021low}, we discuss pole selection for Markov functions, i.e., functions that can be represented as
\begin{equation}\label{eqn:Markov_fun}
    f(z)=\int_{\alpha}^{\beta} \frac{\mathrm{d}\mu(x)}{z-x}
\end{equation}
for some positive measure $\mu(x)$ and $-\infty \le \alpha < \beta < \infty.$ Important examples of functions in this class are
\begin{equation*}
    \frac{\log(1 + z)}{z}=\int_{-\infty}^{-1}\frac{-1/x}{z-x}\,\mathrm{d}x \quad \,\mathrm { and } \,\quad z^\gamma=\frac{\sin(\pi\gamma)}{\pi}\int_{-\infty}^{0}\frac{|x|^\gamma}{z-x}\,\mathrm{d}x,
\end{equation*}
with $-1<\gamma<0$.

Now, let $f$ be a Markov function~\eqref{eqn:Markov_fun} with $\beta\le 0$ and let $A$ be a symmetric positive definite HSS matrix, with the smallest and largest eigenvalues given by $a>0$ and $b\ge a$, respectively.
The quasi-optimal rational approximation of $f$ has been discussed in, e.g.,~\cite[Section~6.2]{beckermann2009error}, which for every $k$ provides a list of poles $\vec \xi_k$ such that 
\begin{equation*}
    \min_{r\in \pol_{k}/q_k}\norm{f-r}_\infty\le 4\norm{f}_{\infty}\exp\left(-k\frac {\pi^2}{\log(16b/a)}\right).
\end{equation*}
Hence, to achieve accuracy $\epsilon$ in Algorithm~\ref{Alg:main2}, one can choose
\begin{equation}\label{eqn:bound-markov}
    k\ge \log\big( 16L\norm{f}_{\infty} \epsilon^{-1} \big) \log(16b/a) / \pi^2.
\end{equation}
Assuming a polynomial growth of the condition number $b/a$ with respect to $n$, the complexity of the algorithm for Markov functions is, according to~\eqref{eq:complexity}, given by
\[ \mathcal O( n \log^2 n (\log \log n + \log \epsilon^{-1} + \log \norm{f}_{\infty} )^2 r^2). \] 
For example, the fixed-accuracy approximation to the inverse square root $A^{-1/2}$ requires nearly $\mathcal O( n \log^4 n)$ operations, when ignoring the $\log \log n$ factor and assuming the HSS rank $r$ of $A$ to be constant.

\section{Numerical experiments}\label{sec:Numerical_exp}

We have implemented Algorithm~\ref{Alg:main2} in {\sc Matlab} and have made the code freely accessible at \url{https://github.com/numpi/HSS-matfun}; this implementation will be denoted by \CKR in the following.  In our implementation, we allow for variable HSS/telescopic ranks (see Remark~\ref{remark:HSSranks}) and employ deflation criteria in the computation of orthonormal bases for rational Krylov subspaces, removing vectors that after the (re)orthogonalization step have a norm smaller than a prescribed tolerance, proportional to the required accuracy.
The threshold size $t$ of the employed HSS matrices is fixed to $256$. In our experiments, all standard operations with HSS matrices, such as matrix-vector products, have been performed using the hm-toolbox \cite{massei2020hm}. In the tables presented in this section, columns with the caption ``err'' denote the relative error in the Frobenius norm, compared with the result computed by a standard dense solver. Columns with the caption ``time'' report the observed execution time in seconds. All experiments have been executed on a server with two Intel(R) Xeon(R) E5-2650v4 CPU running at 2.20 GHz and 256 GB of RAM, using {\sc Matlab} R2021a with the Intel(R) Math Kernel Library Version 2019.0.3.

The main competitor, denoted by \CKM, is the algorithm developed in \cite{cortinovis2022divide}, in which the authors use the HSS structure of $A$ to perform a divide and conquer method for the computation of $f(A)$. The algorithm computes rational Krylov subspaces associated with (possibly large) HSS matrices, exploiting the structure in solving linear systems. The algorithm can also monitor the variation of the norm of the solution when a new pole is employed; this quantity can be used to stop the procedure if the desired accuracy is reached. We utilized an implementation of this algorithm available at \url{https://github.com/Alice94/MatrixFunctions-Banded-HSS}.

\subsection{Computation of the inverse}

Algorithm~\ref{Alg:main2} is an attractive method for computing the inverse of a symmetric positive definite HSS matrix. By Theorem~\ref{thm:error_bound}, this algorithm returns the exact inverse (at least in exact arithmetic) when employing only one zero pole.
We have tested \CKR in this situation for two different matrices.
In Table~\ref{table:invA-increasing_n}, we report the results for the inversion of the discretized Laplacian, that is,
\begin{equation}\label{eqn:Laplacian}
    A=-\frac 1 {h^2}\begin{bmatrix}
        -2 & 1 \\   1 &\ddots& \ddots \\ & \ddots & \ddots & 1 \\ & & 1 & -2
    \end{bmatrix}\in \R^{n\times n},
\end{equation}
where $h=\frac 1 {n+1}$. In Table~\ref{table:invA-banded}, we show the results obtained when inverting a more general HSS matrix (which is not banded) given by 
the Gr\"unwald-Letnikov 
finite difference discretization of the symmetric 
fractional derivative operator 
$L_{\alpha} := \frac{\partial^\alpha}{\partial x^{\alpha}}$~\cite{meerschaert2004finite}
for $\alpha = 1.5$. In contrast to~\eqref{eqn:Laplacian}, this finite 
difference approximation does not yield a sparse matrix, coherently with 
the non-local properties of fractional differential operators. It can be proven that the matrix can be 
approximated in the HSS format \cite{massei2019fast}
with an HSS rank $\mathcal O(\log n)$.

Additionally to \CKM, we also compare to the randomized algorithm introduced by Levitt and Martinsson in \cite{levitt2022linear} (denoted by {\tt LM}) based on the solution of a small number of linear systems involving $A$ and the {\tt inv} procedure for HSS matrices implemented in the hm-toolbox which is based on the ULV factorization described in~\cite{chandrasekaran2006fast} and explained in \cite[Section~4.3]{massei2020hm}.  
The HSS ranks are calculated using the \texttt{hssrank} command from \cite{massei2020hm}, employing the default tolerance of $10^{-12}$.
\begin{table}[ht]
    
    \centering
        {\begin{tabular}{|c|ccccc|}
            \hline

            $n$ & time \CKR & time \CKM & time {\tt LM} & time {ULV} & time Dense\\ 
            \hline 
            1024 & 0.09 & 0.89 & 0.30 & 0.32 & 0.02 \\ 
            2048 & 0.11 & 1.21 & 0.33 & 0.35 & 0.08  \\ 
            4096 & 0.14 & 2.72 & 0.71 & 0.63 & 0.30  \\ 
            8192 & 0.27 & 6.94 & 2.03 & 0.91 & 1.12  \\ 
            16384 & 0.66 & 16.59 & 4.84 & 1.84 &   \\ 
            32768 & 1.18 & 37.91 & 11.98 & 4.20 &    \\ 
            
\hline
 \end{tabular}
 \bigskip

 \begin{tabular}{|c|cccc|}
    \hline

    $n$ & err \CKR & err \CKM & err {\tt LM} & err {ULV}\\ 
    \hline 
   1024 & 7.56e-13 & 7.74e-12 & 7.98e-12 & 8.03e-12 \\ 
   2048 & 6.15e-13 & 2.75e-12 & 7.29e-12 & 7.26e-12 \\ 
   4096 & 9.47e-12 & 5.86e-12 & 3.43e-12 & 3.99e-12\\ 
   8192 & 8.15e-12 & 9.59e-11 & 1.58e-10 & 1.58e-10  \\    
\hline
\end{tabular}}
\caption{Comparison of the newly proposed algorithm \CKR with \CKM, {\tt LM}, and the {\tt inv} command of the hm-toolbox based on the ULV decomposition, for computing $A^{-1}$, where $A$ is the discretized Laplacian~\eqref{eqn:Laplacian}. } \label{table:invA-increasing_n}
\end{table}

\begin{table}[t]
    \centering
        {\begin{tabular}{|cc|ccccc|}
            \hline

            $n$ & HSS rank $A$ & time \CKR & time \CKM & time {\tt LM} & time ULV & time Dense\\ 
            \hline 
            1024 & 29 & 0.11 & 0.65 & 0.31 & 0.26 & 0.08 \\ 
2048 & 32 & 0.21 & 0.84 & 0.44 & 0.38 & 0.37 \\ 
4096 & 35 & 0.36 & 2.17 & 1.06 & 0.63 & 1.72 \\ 
8192 & 37 & 0.62 & 5.48 & 2.14 & 1.41 & 11.20 \\ 
            
\hline
 \end{tabular}
 \bigskip

 \begin{tabular}{|cc|cccc|}
    \hline

    $n$ & HSS rank $A$& err \CKR & err \CKM & err {\tt LM} & err {ULV}\\ 
    \hline 
   1024 &29& 5.50e-13 & 1.11e-12 & 4.16e-13 &7.24e-12 \\ 
   2048 &32& 4.78e-13 & 3.26e-12 & 1.18e-12 &2.11e-11 \\ 
   4096 &35& 2.08e-12 & 2.75e-11 & 2.53e-12 &6.81e-11\\ 
   8192 &37& 4.99e-12 & 4.36e-11 & 9.68e-12 &1.75e-10 \\    
\hline
\end{tabular}}
\caption{
Comparison of the newly proposed algorithm \CKR with \CKM, {\tt LM}, and the {\tt inv} command of the hm-toolbox based on the ULV decomposition, for computing $A^{-1}$, 
where $A$ is the Gr\"unwald-Letnikov finite difference discretization of the 
fractional derivative of order $\alpha = 1.5$~\cite{massei2019fast,meerschaert2004finite}.} \label{table:invA-banded}
\end{table}

Although not specifically designed for matrix inversion, \CKR is always the fastest among the methods that exploit HSS structure, while attaining a comparable level of accuracy. Even the closest competitor  ULV is significantly slower, by up to a factor $3$--$4$.

\subsection{Computation of the exponential function}

To show the effectiveness of rational approximation of the exponential function, in Table~\ref{table:expm-increasing-a} we compute the matrix exponential of a tridiagonal matrix $A$, whose eigenvalues are uniformly distributed in $[-10^a,0]$, for different values of $a$. For the computation, we compare the presented method with optimal poles and \CKM with both optimal and infinity poles, in the latter case the built-in stopping criteria are employed. In Table~\ref{table:expm-laplacian} we also report the comparison between the presented method and the {\tt expm} function implemented in the hm-toolbox for the computation of $\mathrm{exp}(A)$ based on the Padè approximant, where $A$ is the discretized Laplacian defined in \eqref{eqn:Laplacian}.
\begin{table}[ht]
    \centering {
\begin{tabular}{|c|ccc|ccc|}
    \hline
    a &time \CKR &time \CKM Poly &time \CKM Rat& err \CKR & err \CKM Poly & err \CKM Rat  \\ 
    \hline 
    $0$ & 0.82 & 1.38 & 13.06& 1.04e-11 & 2.16e-10 & 1.52e-10 \\ 
    $2$ & 0.67 & 1.31 & 11.15& 2.89e-10 & 1.75e-07 & 6.11e-09 \\ 
    $4$ & 0.76 & 0.48 & 10.32& 2.50e-12 & 1.05e-03 & 1.20e-08 \\ 
    $6$ & 0.52 & 0.52 & 10.13 & 2.84e-10 & 1.03e-02 & 4.69e-11 \\ 
    $8$& 0.53 & 0.48 & 10.14  & 3.36e-08 & 1.21e+01 & 5.28e-08 \\ 
\hline
\end{tabular}}
\caption{Computation of the matrix exponential of a matrix of size 4096, whose eigenvalues are uniformly distributed in $[-10^a,0]$, for different values of $a$. The accuracy is set to $10^{-8}.$} \label{table:expm-increasing-a}
\end{table}

\begin{table}[ht]
    \centering {
\begin{tabular}{|c|cccc|}
    \hline
    n &time \CKR &time {\tt expm} & err \CKR & err {\tt expm} \\ 
    \hline 
    $1024$  & 0.33 & 2.12 & 4.58e-10 & 5.35e-04\\ 
    $2048$  & 0.59 & 3.45 & 2.01e-09 & 2.14e-03\\ 
    $4096$  & 1.02 & 7.01 & 6.16e-09 & 8.52e-03\\ 
    $8192$  & 2.03 & 14.87 & 2.75e-08 & 3.37e-02\\ 
\hline
\end{tabular}}
\caption{computation of $\mathrm{exp}(A)$ where $A$ is the discretization of the Laplacian defined in \eqref{eqn:Laplacian} of $n$ using the presented method and the routine {\tt expm} of the hm-toolbox. The accuracy is set to $10^{-8}.$} \label{table:expm-laplacian}
\end{table}

Again, our newly proposed method \CKR is significantly faster than the competitors, while resulting in comparable accuracy. Note that \CKM Poly appears to not use the correct stopping criterion for larger $a$, resulting in an unacceptably large error.

\subsection{Computation of the inverse square root}
To test the presented algorithm for the computation of the inverse square root of an HSS matrix, we consider the problem of sampling from a Gaussian Markov random field (see \cite[Section~4.2]{cortinovis2022divide}) which reduces to the computation of the inverse of the square root of a banded matrix. In Table~\ref{table:sampl-increasing_n} we compare our algorithm with optimal poles, with \CKM with extended poles (i.e., alternating $0$ and $\infty$); the latter choice of poles is the one made by the authors of \CKM for solving the presented problem: since the algorithm needs to solve possibly large linear systems, the choice of using mutually different poles can often not be the most advantageous strategy. The number of poles to employ in our method is given by \eqref{eqn:bound-markov} (which is in practice very pessimistic) and the accuracy is only used in the determination of the deflation tolerance. The termination of \CKM is due to the built-in stopping criteria. For completeness, we also approximate $f(A)$ by explicitly evaluating a rational approximation of $f$: the poles and the residuals of the rational approximation have been derived using the AAA algorithm \cite{nakatsukasa2018aaa}, and for the evaluation, the HSS structure has been exploited using the hm-toolbox \cite{massei2020hm}. In all the cases reported, the 
degree of the rational approximant constructed by AAA is 
$12$. While \CKR is still faster than \CKM for sufficiently large $n$, its advantage in terms of speed is less evident for this example. Note, however, that its error is significantly lower.

\begin{table}[ht]
    
    \centering
        { \begin{tabular}{|cc|cccc|}
            \hline
            size & HSS rank $A$ & time \CKR & time \CKM & time Rat & time Dense \\
            \hline 
            512 & 22 & 0.13 & 0.22 & 1.28 & 0.03  \\ 
1024 & 20 & 0.29 & 0.30 & 2.38 & 0.23  \\ 
2048 & 21 & 0.57 & 0.73 & 5.09 & 0.87 \\ 
4096 & 21 & 1.24 & 1.55 & 12.82 & 7.04  \\ 
8192 & 23 & 3.36 & 4.07 & 27.32 & 63.07 \\ 
16384 & 25 & 6.90 & 9.17 & &    \\ 
32768 & 28 & 13.83 & 20.05 & &   \\ 
65536 & 24 & 27.21 & 44.85 & &   \\ 
131072 & 27 & 54.50 & 104.01 & &   \\ 
                
    \hline
 \end{tabular}
 \bigskip

 \begin{tabular}{|cc|ccc|}
    \hline

    size & HSS rank $A$ & err \CKR & err \CKM & err Rat   \\ 
    \hline 
   512 & 22 & 6.67e-14 & 2.02e-09 & 1.87e-09    \\ 
   1024 & 20 & 1.32e-13 & 2.70e-09 & 6.52e-09\\
   2048 & 21 & 6.00e-11 & 3.64e-09 & 3.73e-09    \\ 
   4096 & 21 & 1.99e-13 & 3.39e-09 & 4.34e-09 \\ 
   8192 & 23  & 1.11e-13 & 3.72e-09 & 5.52e-09   \\

\hline
\end{tabular}}
 \caption{Comparison of the newly proposed algorithm \CKR (using optimal poles), \CKM with extended Krylov subspaces, and the evaluation of a rational approximation, for the computation of $f(A)$ with accuracy of $10^{-8}$, where $f(x)=1/\sqrt{x}$, and $A$ is the sampling from a Gaussian Markov random field.} \label{table:sampl-increasing_n}
\end{table}

We also show a comparison between \CKR and \CKM, both with (quasi-)optimal poles, for the approximation of the fractional Laplacian i.e., the computation of $(-A)^{-1/2}$, where $A$ is defined in \eqref{eqn:Laplacian}. In Table~\ref{table:frac-laplac-increasing_n} we compare the timing and the relative error between the presented algorithm and \CKM using in both cases $50$ optimal poles and varying the size of $A$. 

\begin{table}[ht]
    \centering
        \begin{tabular}{|c|ccc|cc|}
            \hline
            $n$ & time \CKR & time \CKM & time Dense & err \CKR & err \CKM \\ 
            \hline 
           1024 & 1.44 & 7.15 & 0.19 & 1.32e-11& 2.04e-11 \\ 
           2048 & 1.71 & 15.95 & 1.05 & 1.13e-11& 4.68e-11 \\ 
           4096 & 4.68 & 41.94 & 7.96 & 1.31e-10& 1.77e-10 \\ 
           8192 & 7.87 & 91.89 & & &  \\ 
           16384 & 16.30 & 223.16 & & &  \\
            
\hline
 \end{tabular}
        \caption{Comparison of the newly proposed algorithm \CKR with \CKM, for the computation of $f(A)$ where $f(x)=1/\sqrt{x}$,  and $A$ is the discretization of the Laplacian. In both algorithms $50$ quasi-optimal poles have been employed. } \label{table:frac-laplac-increasing_n}
\end{table}

\subsection{Computation of the sign function}
In this section, we compute $f(A)$ where $f$ is the sign function, i.e.,
\begin{equation*}
    f(x)=\begin{cases}
        \phantom{-}1 \quad   x >0,\\
        -1\quad  x\le 0.
    \end{cases}
\end{equation*}
Assuming that $A$ has both positive and negative eigenvalues (otherwise the computation of $f(A)$ is trivial) the discontinuity of the function does not allow for a reasonable rational approximation on an interval containing the eigenvalues of $A$. In particular, our convergence result from Theorem~\ref{thm:error_bound} does not apply. On the other hand, if the eigenvalues of $A$ are contained in $\mathbb{E}=[-b,-a]\cup [a,b]$, with $a,b,>0$, then the best rational approximation of the sign function on $\mathbb E$ is explicitly known in terms of elliptic functions, see \cite[Section~4.3]{petrushev2011rational}. In Table~\ref{table:sign1024}, we test the time and the accuracy of the proposed method on tridiagonal matrices whose positive eigenvalues are logarithmically distributed in the interval $[10^a,1]$ and the negative ones are given by the symmetrization with respect to the imaginary axis. We compare the results with the ones obtained by running \CKM with optimal poles and with the evaluation of the rational approximation given by the AAA algorithm \cite{nakatsukasa2018aaa}, using the routines contained in the hm-toolbox \cite{massei2020hm}.

While not covered by the theory, \CKR is clearly the best method and attains good accuracy until $a = -7$. For $a = -9$, the accuracy of all methods suffers from the fact that the eigenvalues get too close to zero.

\begin{table}
    \centering
        \begin{tabular}{|c|ccc|ccc|}
            \hline

            a & time \CKR & time \CKM & time Rat &  err \CKR & err \CKM & err Rat \\ 
            \hline 
-1 & 1.73 & 20.55 & 50.28 & 3.75e-10& 3.72e-10 & 1.73e-09  \\ 
-3 & 4.13 & 38.46 & 128.93 & 2.60e-10& 4.10e-08 & 1.51e-08  \\ 
-5 & 9.72 & 57.99 & 121.76 & 2.70e-10& 2.12e-06 & 7.49e-08  \\ 
-7 & 18.24 & 78.37 & 137.72 & 2.99e-08& 1.79e-08 & 1.51e-05  \\  
-9 & 14.08 & 43.04 & 139.65 & 7.45e-02& 7.83e-02 & 3.98e-02  \\ 
\hline
 \end{tabular}
        \caption{Computation of $\mathrm{sign}(A)$, where $A$ is a tridiagonal matrix of size 4096 with logarithmically spaced eigenvalues, symmetric with respect to the imaginary axis contained in $[-1,-10^a]\cup [10^a,1]$. } \label{table:sign1024}
\end{table}

\section{Conclusions}

Generalizing the definition of telescopic decompositions, we have linked different representations of HSS matrices used in the literature, providing ways to convert between them. Exploiting the nested low-rank structure of telescopic decompositions, we have developed a novel algorithm that computes an approximation of $f(A)$ for a symmetric HSS matrix $A$.
Our convergence results imply nearly linear complexity for matrix exponentials and linear-polylogarithmic complexity for inverse square roots in situations of practical relevance.
This favorable complexity is attained by using rational Krylov subspaces that involve small-sized matrices only, avoiding the solution of potentially large linear systems usually associated with rational Krylov subspace techniques. Several numerical experiments show that our newly proposed algorithm is faster than existing algorithms for a variety of examples previously reported in the literature. Somewhat surprisingly, it even appears to be the method of choice for computing matrix inverses. A number of questions remain open. This includes the extension to nonsymmetric matrices as well as a theoretical explanation of the good results obtained for the sign function.

\section*{Acknowledgments}
The first author is grateful to the ANCHP group of the Ecole
Polytechnique Fédérale de Lausanne for their warm hospitality. The first and third authors are both members of the research group INdAM-GNCS.
The third author has been supported by 
the National Research Center in High Performance Computing, Big Data and Quantum Computing (CN1 -- Spoke 6), and acknowledges the MIUR Excellence Department Project awarded to the Department of Mathematics, University of Pisa, CUP I57G22000700001.

\bibliographystyle{plain}
    \bibliography{biblio}

\end{document}